\documentclass{amsart}
\usepackage{amsmath,amsthm,amsfonts,amssymb}

\numberwithin{equation}{section}
\theoremstyle{plain}

\newtheorem{theorem}{Theorem}

\newtheorem{lemma}[theorem]{Lemma}

\newtheorem{remark}[theorem]{Remark}

\begin{document}

\title[Process-Level LDP for Nonlinear Hawkes Point Processes]{Process-Level Large Deviations for Nonlinear Hawkes Point Processes}
\author{LINGJIONG ZHU}
\address
{Courant Institute of Mathematical Sciences\newline
\indent New York University\newline
\indent 251 Mercer Street\newline
\indent New York, NY-10012\newline
\indent United States of America}
\email{ling@cims.nyu.edu}
\date{9 August 2011. \textit{Revised:} 13 October 2012}

\thanks{This research was supported partially by a grant from the National Science Foundation: DMS-0904701, DARPA grant
and MacCracken Fellowship at NYU}

\begin{abstract}
In this paper, we prove a process-level, also known as level-3 large deviation principle for a very general class
of simple point processes, i.e. nonlinear Hawkes process, with a rate function given by the process-level entropy, which
has an explicit formula.
\end{abstract}

\maketitle

\section{Introduction and Main Results}

In this paper, we study the process-level large deviations for a very general class of point processes,
i.e. nonlinear Hawkes processes. The rate function is given by the process-level entropy, which has an explicit formula
via the Girsanov theorem for two absolutely continuous point processes. Our methods and ideas should work for
some other point processes as well and we would expect the same expression for the rate function.

Let $N$ be a simple point process on $\mathbb{R}$ and let $\mathcal{F}^{-\infty}_{t}:=\sigma(N(C),C\in\mathcal{B}(\mathbb{R}), C\subset(-\infty,t])$ be
an increasing family of $\sigma$-algebras. Any nonnegative $\mathcal{F}^{-\infty}_{t}$-progressively measurable process $\lambda_{t}$ with
\begin{equation}
\mathbb{E}\left[N(a,b]|\mathcal{F}^{-\infty}_{a}\right]=\mathbb{E}\left[\int_{a}^{b}\lambda_{s}ds\big|\mathcal{F}^{-\infty}_{a}\right]
\end{equation}
a.s. for all intervals $(a,b]$ is called an $\mathcal{F}^{-\infty}_{t}$-intensity of $N$. We use the notation $N_{t}:=N(0,t]$ to denote the number of
points in the interval $(0,t]$. A point process $Q$ is simple if $Q(\exists t: N[t-,t]\geq 2)=0$.

A general Hawkes process is a simple point process $N$ admitting an $\mathcal{F}^{-\infty}_{t}$-intensity
\begin{equation}
\lambda_{t}:=\lambda\left(\int_{-\infty}^{t}h(t-s)N(ds)\right),\label{dynamics}
\end{equation}
where $\lambda(\cdot):\mathbb{R}^{+}\rightarrow\mathbb{R}^{+}$ is locally integrable, left continuous, 
$h(\cdot):\mathbb{R}^{+}\rightarrow\mathbb{R}^{+}$ and
we always assume that $\Vert h\Vert_{L^{1}}=\int_{0}^{\infty}h(t)dt<\infty$. 
In \eqref{dynamics}, $\int_{-\infty}^{t}h(t-s)N(ds)$ stands for $\int_{(-\infty,t)}h(t-s)N(ds)=\sum_{\tau<t}h(t-\tau)$, where
$\tau$ are the occurences of the points before time $t$. Since we assume that $\lambda(\cdot)$ is left continuous,
the intensity $\lambda_{t}$ is predictable. 

In the literatures, $h(\cdot)$ and $\lambda(\cdot)$ are usually referred to
as exciting function and rate function respectively.

A Hawkes process is linear if $\lambda(\cdot)$ is linear and it is nonlinear otherwise. The linear Hawkes process
was introduced in Hawkes \cite{Hawkes} and the nonlinear Hawkes process was introduced in 
Br\'{e}maud and Massouli\'{e} \cite{Bremaud}.

The Hawkes process captures both the self-exciting property and the clustering effect. 
You can think of the arrival times $\tau$ as ``bad'' events, which can be the arrivals of claims in insurance literature or 
the time of defaults of big firms in the real world. Hawkes process has many applications in finance. 
It is used for the calculation of conditional risk measures, ruin probabilities and credit default modeling. 
For a list of references to the applications in finance, see Liniger \cite{Liniger}.

Hawkes processes may also have applications in cosmology, ecology, epidemiology, seismological applications, 
neuroscience applications and DNA modeling. For a list of references to these applications, see Bordenave and Torrisi \cite{Bordenave}.

For a short history of Hawkes processes, we refer to Liniger \cite{Liniger}.

When $\lambda(\cdot)$ is nonlinear, Br\'{e}maud and Massouli\'{e} \cite{Bremaud} proves that 
under certain conditions, there exists a unique stationary version of the nonlinear Hawkes process and 
Br\'{e}maud and Massouli\'{e} \cite{Bremaud} also proves the convergence to equilibrium of a nonstationary version, 
both in distribution and in variation. 

Throughout this paper, we assume that
\begin{itemize}
\item 
The exciting function $h(t)$ is positive, continuous and decreasing for $t\geq 0$ and $h(t)=0$ for any $t<0$. 
We also assume that $\int_{0}^{\infty}h(t)dt<\infty$.

\item 
The rate function $\lambda(\cdot):[0,\infty)\rightarrow\mathbb{R}^{+}$ is increasing and $\lim_{z\rightarrow\infty}\frac{\lambda(z)}{z}=0$. 
We also assume that $\lambda(\cdot)$ is Lipschitz with constant $\alpha>0$, i.e. $|\lambda(x)-\lambda(y)|\leq\alpha|x-y|$ for any $x,y\geq 0$.
\end{itemize}

Because of the flexibility of $\lambda(\cdot)$ and $h(\cdot)$, this will give us a very wide class of 
simple point processes. Later, if you go through the proofs of the process-level large deviation principle
in our paper, you can see that if for any simple point process you want to obtain a process-level large deviation principle,
it has to satisfy some regularities like the assumptions in our paper. We refer to Section \ref{conclusion} for detailed discussions.

Let $\Omega$ be the set of countable, locally finite subsets of $\mathbb{R}$ and for any $\omega\in\Omega$ 
and $A\subseteq\mathbb{R}$ and write $\omega(A):=\omega\cap A$. For any $t\in\mathbb{R}$, we write $\omega(t)=\omega(\{t\})$. 
Let $N(A)=\#|\omega\cap A|$ denote the number of points in the set $A$ for any $A\subset\mathbb{R}$. 
We also use the notation $N_{t}$ to denote $N[0,t]$, the number of points up to time $t$ starting from time $0$. 
We define the shift operator $\theta_{t}$ by $\theta_{t}(\omega)(s)=\omega(t+s)$. We equip the sample space $\Omega$ 
with the topology in which the convergence $\omega_{n}\rightarrow\omega$ as $n\rightarrow\infty$ is defined by
\begin{equation}
\sum_{\tau\in\omega_{n}}f(\tau)\rightarrow\sum_{\tau\in\omega}f(\tau),
\end{equation}
for any continuous $f$ with compact support.

This topology is equivalent to the vague topology for random measures. For a discussion on vague topology, 
random measures and point processes, see for example Grandell \cite{Grandell}. One can equip the space of 
locally finite random measures with the vague topology. The subspace of integer valued random measures is 
then the space of point processes. A simple point processes is a point process without multiple jumps. 
The space of point processes is closed. But the space of simple point processes is not closed. 

Denote $\mathcal{F}^{s}_{t}=\sigma(\omega[s,t])$ for any $s<t$, i.e. the $\sigma$-algebra generated by 
all the possible configurations of points in the interval $[s,t]$. Denote $\mathcal{M}(\Omega)$ the space of 
probability measures on $\Omega$. We also define $\mathcal{M}_{S}(\Omega)$ as the space of simple point processes 
that are invariant with respect to $\theta_{t}$ with bounded first moment, i.e. for any $Q\in\mathcal{M}_{S}(\Omega)$,
$\mathbb{E}^{Q}[N[0,1]]<\infty$.
Define $\mathcal{M}_{E}(\Omega)$ as the set of ergodic simple point processes in $\mathcal{M}_{S}(\Omega)$. 
We define the topology of $\mathcal{M}_{S}(\Omega)$ as the following.
For a sequence $Q_{n}$ in $\mathcal{M}_{S}(\Omega)$ and $Q\in\mathcal{M}_{S}(\Omega)$, we say
$Q_{n}\rightarrow Q$ as $n\rightarrow\infty$ if and only if
\begin{equation}
\int fdQ_{n}\rightarrow\int fdQ,
\end{equation}
as $n\rightarrow\infty$ for any continuous and bounded $f$ and
\begin{equation}
\int N[0,1](\omega)Q_{n}(d\omega)\rightarrow\int N[0,1](\omega)Q(d\omega),
\end{equation}
as $n\rightarrow\infty$. In other words, the topology is the strengthened weak topology with the convergence of the first moment of $N[0,1]$.
For any $Q_{1}$, $Q_{2}$ in $\mathcal{M}_{S}(\Omega)$, one can define the metric $d(\cdot,\cdot)$ by
\begin{equation}
d(Q_{1},Q_{2})=d_{p}(Q_{1},Q_{2})+\left|\mathbb{E}^{Q_{1}}[N[0,1]]-\mathbb{E}^{Q_{2}}[N[0,1]]\right|,
\end{equation}
where $d_{p}(\cdot,\cdot)$ is the usual Prokhorov metric.
Because this is an unusual topology, the compactness is different
from that in the usual weak topology and also, later, when we prove the exponential tightness, we need to take some extra care.
See Lemma \ref{tightness} and (iii) of Lemma \ref{finalestimate}.

We denote by $C(\Omega)$ the set of real-valued continous functions on $\Omega$. We similarly define $C(\Omega\times\mathbb{R})$.
We also denote by $\mathcal{B}(\mathcal{F}^{-\infty}_{t})$ the set of all bounded $\mathcal{F}^{-\infty}_{t}$
progressively measurable and $\mathcal{F}^{-\infty}_{t}$ predictable functions.

Before we proceed, recall that a sequence $(P_{n})_{n\in\mathbb{N}}$ of probability measures on a topological space $X$ 
satisfies the large deviation principle (LDP) with rate function $I:X\rightarrow\mathbb{R}$ if $I$ is non-negative, 
lower semicontinuous and for any measurable set $A$,
\begin{equation}
-\inf_{x\in A^{o}}I(x)\leq\liminf_{n\rightarrow\infty}\frac{1}{n}\log P_{n}(A)
\leq\limsup_{n\rightarrow\infty}\frac{1}{n}\log P_{n}(A)\leq-\inf_{x\in\overline{A}}I(x).
\end{equation}
Here, $A^{o}$ is the interior of $A$ and $\overline{A}$ is its closure. See Dembo and Zeitouni \cite{Dembo} or Varadhan \cite{Varadhan} 
for general background regarding large deviations and their applications. 
Also Varadhan \cite{VaradhanII} has an excellent survey article on this subject.

Let us have a brief review on what is known about large deviations for Hawkes processes.

When $\lambda(\cdot)$ is linear, say $\lambda(z)=\nu+z$, then one can use immigration-birth representation, 
also known as Galton-Watson theory to study it. Under the immigration-birth representation, 
if the immigrants are distributed as Poisson process with intensity $\nu$ and each immigrant generates 
a cluster whose number of points is denoted by $S$, then $N_{t}$ is the total number of points generated 
in the clusters up to time $t$. If the process is ergodic, we have
\begin{equation}
\lim_{t\rightarrow\infty}\frac{N_{t}}{t}=\nu\mathbb{E}[S],\quad\text{a.s.}
\end{equation}
In particular, for linear Hawkes process with rate function $\lambda(z)=\nu+z$ and exciting function $h(\cdot)$ such 
that $\Vert h\Vert_{L^{1}}<1$, we have
\begin{equation}
\lim_{t\rightarrow\infty}\frac{N_{t}}{t}=\mu,\quad\text{a.s.},\quad\text{where $\mu:=\frac{\nu}{1-\Vert h\Vert_{L^{1}}}$.}
\end{equation}
Recently, Bacry et al. \cite{Bacry} proved a functional central limit theorem for linear multivariate Hawkes processes under certain assumptions.
That includes the linear Hawkes process as a special case and they proved that
\begin{equation}
\frac{N_{\cdot t}-\cdot\mu t}{\sqrt{t}}\rightarrow\sigma B(\cdot),\quad\text{as $t\rightarrow\infty$,}
\end{equation}
where $B(\cdot)$ is a standard Brownian motion. The convergence
is weak convergence on $D[0,1]$, the space of c\'{a}dl\'{a}g functions on $[0,1]$, equipped with Skorokhod topology.

Bordenave and Torrisi \cite{Bordenave} proves that if $0<\int_{0}^{\infty}h(t)dt<1$ and $\int_{0}^{\infty}th(t)dt<\infty$, 
then $\left(\frac{N_{t}}{t}\in\cdot\right)$ satisfies the large deviation principle with the good rate function
\begin{equation}
I(x)=
\begin{cases}
x\log\left(\frac{x}{\nu+x\Vert h\Vert_{L^{1}}}\right)-x+x\Vert h\Vert_{L^{1}}+\nu &\text{if $x\in[0,\infty)$}
\\
+\infty &\text{otherwise}
\end{cases}.
\end{equation}

When $\lambda(\cdot)$ is linear, Zhu \cite{Zhu} gives an alternative proof for the large deviation principle for $(N_{t}/t\in\cdot)$.

Once the LDP for $\left(\frac{N_{t}}{t}\in\cdot\right)$ is established, it is easy to study the ruin probability. 
Stabile and Torrisi \cite{Stabile} considered risk processes with non-stationary Hawkes claims arrivals and studied 
the asymptotic behavior of infinite and finite horizon ruin probabilities under light-tailed conditions on the claims.

When $\lambda(\cdot)$ is nonlinear, Br\'{e}maud and Massouli\'{e} \cite{Bremaud}
studied the stability results and once you have ergodicity, the ergodic theorem
automatically implies the law of large numbers. Recently, Zhu \cite{ZhuII} proved
a functional central limit theorem for nonlinear Hawkes process.

For the LDP for nonlinear Hawkes processes, \cite{Zhu} obtained the large deviation results for $(N_{t}/t\in\cdot)$ 
for some special cases. He proved the case when $h(\cdot)$ is exponential first, then generalizes the proof to the case 
when $h(\cdot)$ is a sum of exponentials and finally uses that to prove the LDP for a special class of general Hawkes processes.
The Hawkes process is interesting, partly because it is not Markovian unless $h(\cdot)$ is exponential or sum of exponentials.
The methods in Zhu \cite{Zhu} relies on proving the LDP for Markovian case first and then approximate the general
case, i.e. general $h(\cdot)$ using the Markovian case. But it is very difficult to do so. Indeed, Zhu \cite{Zhu}
can only prove LDP for general $h(\cdot)$ when $\lim_{z\rightarrow\infty}\frac{\lambda(z)}{z^{\alpha}}=0$ for any $\alpha>0$.
Therefore, it is natural in our paper to consider proving the level-3 large deviation principle first and then use
the contraction principle to obtain the LDP for $(N_{t}/t\in\cdot)$. We also want to point out that in the case when $\lambda(\cdot)$
is linear, even if $h(\cdot)$ is general, the LDP for $(N_{t}/t\in\cdot)$ can still be established because the linear case can
be explicitly computed, which is not the case when $\lambda(\cdot)$ is nonlinear.

In the pioneering work by Donsker and Varadhan \cite{Donsker}, they obtained a level-3 large deviation result for 
certain stationary Markov processes. 

We would like to prove the large deviation principle for general Hawkes processes by proving a process-level, also known as 
level-3 large deviation principle first. We can then use the contraction principle to obtain the level-1 large deviation principle for $(N_{t}/t\in\cdot)$.

Let us define the empirical measure for the process as
\begin{equation}
R_{t,\omega}(A)=\frac{1}{t}\int_{0}^{t}\chi_{A}(\theta_{s}\omega_{t})ds,
\end{equation}
for any $A$, where $\omega_{t}(s)=\omega(s)$ for $0\leq s\leq t$ and $\omega_{t}(s+t)=\omega_{t}(s)$ for any $s$. 
Donsker and Varadhan \cite{Donsker} proved that in the case when $\Omega$ is a space of c\`{a}dl\`{a}g functions $\omega(\cdot)$ 
on $-\infty<t<\infty$ endowed with Skorohod topology and taking values in a Polish space $X$, under certain conditions, 
$P^{0,x}(R_{t,\omega}\in\cdot)$ satisfies a large deviation principle, where $P^{0,x}$ is a Markov process on $\Omega^{0}_{\infty}$ 
with initial value $x\in X$. The rate function $H(Q)$ is some entropy function.

Let $h(\alpha,\beta)_{\Sigma}$ be the relative entropy of $\alpha$ with respect to $\beta$ restricted to the $\sigma$-algebra $\Sigma$. 
For any $Q\in\mathcal{M}_{S}(\Omega)$, let $Q^{\omega^{-}}$ be the regular conditional probability distribution of $Q$. 
Similarly we define $P^{\omega^{-}}$.

Let us define the entropy function $H(Q)$ as
\begin{equation}
H(Q)=
\mathbb{E}^{Q}[h(Q^{\omega^{-}},P^{\omega^{-}})_{\mathcal{F}^{0}_{1}}].
\end{equation}
Notice that $P^{\omega^{-}}$ is the Hawkes process conditional on the past history $\omega^{-}$. 
It has rate $\lambda(\sum_{\tau\in\omega[0,s)\cup\omega^{-}}h(s-\tau))$ at time $0\leq s\leq 1$, 
which is well defined for almost every $\omega^{-}$ under $Q$ if $\mathbb{E}^{Q}[N[0,1]]<\infty$ 
since $\mathbb{E}^{Q}[\sum_{\tau\in\omega^{-}}h(-\tau)]=\Vert h\Vert_{L^{1}}\mathbb{E}^{Q}[N[0,1]]<\infty$ 
implies $\sum_{\tau\in\omega^{-}}h(s-\tau)\leq\sum_{\tau\in\omega^{-}}h(-\tau)<\infty$ for all $0\leq s\leq 1$.

When $H(Q)<\infty$, $h(Q^{\omega^{-}},P^{\omega^{-}})<\infty$ for a.e. $\omega^{-}$ under $Q$, which implies that 
$Q^{\omega^{-}}\ll P^{\omega^{-}}$ on $\mathcal{F}^{0}_{1}$. By the theory of absolute continuity of point processes, 
see for example Chapter 19 of Lipster and Shiryaev \cite{Lipster} or Chapter 13 of Daley and Vere-Jones \cite{Daley}, 
the compensator of $Q^{\omega^{-}}$ is absolutely continuous, i.e. it has some density $\hat{\lambda}$ say, such that by the Girsanov formula,
\begin{align}
H(Q)&=\int_{\Omega^{-}}\int\left[\int_{0}^{1}\left(\lambda-\hat{\lambda}\right)ds+\int_{0}^{1}
\log(\hat{\lambda}/\lambda)dN_{s}\right]dQ^{\omega^{-}}Q(d\omega^{-})\label{HFunction}
\\
&=\int_{\Omega}\left[\int_{0}^{1}\lambda(\omega,s)-\hat{\lambda}(\omega,s)
+\log\left(\frac{\hat{\lambda}(\omega,s)}{\lambda(\omega,s)}\right)\hat{\lambda}ds\right]Q(d\omega)\nonumber,
\end{align}
where $\lambda=\lambda\left(\sum_{\tau\in\omega[0,s)\cup\omega^{-}}h(s-\tau)\right)$. Both $\lambda$ and $\hat{\lambda}$ 
are $\mathcal{F}^{-\infty}_{s}$-predictable for $0\leq s\leq 1$. For the equality in \eqref{HFunction}, we 
used the fact that $N_{t}-\int_{0}^{t}\hat{\lambda}(\omega,s)ds$ is a martingale under $Q$ and
for any $f(\omega,s)$ which is bounded, $\mathcal{F}^{-\infty}_{s}$ progressively measurable and predictable, we have
\begin{equation}
\int_{\Omega}\int_{0}^{1}f(\omega,s)dN_{s}Q(d\omega)=\int_{\Omega}\int_{0}^{1}f(\omega,s)\hat{\lambda}(\omega,s)dsQ(d\omega).
\end{equation}
We will use the above fact repeatedly in our paper.

The following theorem is the main result of this paper.

\begin{theorem}
For any open set $G\subset\mathcal{M}_{S}(\Omega)$,
\begin{equation}
\liminf_{t\rightarrow\infty}\frac{1}{t}\log P\left(R_{t,\omega}\in G\right)\geq -\inf_{Q\in G}H(Q),
\end{equation}
and for any closed set $C\subset\mathcal{M}_{S}(\Omega)$,
\begin{equation}
\limsup_{t\rightarrow\infty}\frac{1}{t}\log P\left(R_{t,\omega}\in C\right)\leq -\inf_{Q\in C}H(Q).
\end{equation}
\end{theorem}

We will prove the lower bound in Section \ref{lowerbound}, the upper bound in Section \ref{upperbound} and the 
superexponential estimates that are needed in the proof of the upper bound in Section \ref{superexpestimates}.

Once we establish the level-3 large deviation result, we can obtain the large deviation principle for $(N_{t}/t\in\cdot)$ 
directly by using the contraction principle.

\begin{theorem}
$(N_{t}/t\in\cdot)$ satisfies a large deviation principle with the rate function $I(\cdot)$ given by
\begin{equation}
I(x)=\inf_{Q\in\mathcal{M}_{S}(\Omega),\mathbb{E}^{Q}[N[0,1]]=x}H(Q).
\end{equation}
\end{theorem}

\begin{proof}
Since $Q\mapsto\mathbb{E}^{Q}[N[0,1]]$ is continuous, $\int_{\Omega}N[0,1]dR_{t,\omega}$ satisfies a large deviation principle 
with the rate function $I(\cdot)$ by the contraction principle. (For a discussion on contraction principle, 
see for example Varadhan \cite{Varadhan}.)
\begin{align}
\int_{\Omega}N[0,1]dR_{t,\omega}&=\frac{1}{t}\int_{0}^{t}N[0,1](\theta_{s}\omega_{t})ds
\\
&=\frac{1}{t}\int_{0}^{t-1}N[s,s+1](\omega)ds+\frac{1}{t}\int_{t-1}^{t}N[s,s+1](\omega_{t})ds.\nonumber
\end{align}
Notice that
\begin{equation}
0\leq\frac{1}{t}\int_{t-1}^{t}N[s,s+1](\omega_{t})ds\leq\frac{1}{t}(N[t-1,t](\omega)+N[0,1](\omega)),
\end{equation}
and
\begin{equation}
\frac{1}{t}\int_{0}^{t-1}N[s,s+1](\omega)ds=\frac{1}{t}\left[\int_{t-1}^{t}N_{s}(\omega)ds-\int_{0}^{1}N_{s}(\omega)ds\right]
\leq\frac{N_{t}}{t},
\end{equation}
and $\frac{1}{t}\int_{0}^{t-1}N[s,s+1](\omega)ds\geq\frac{N_{t-1}-N_{1}}{t}=\frac{N_{t}}{t}-\frac{N[t-1,t]+N_{1}}{t}$. Hence,
\begin{equation}
\frac{N_{t}}{t}-\frac{N[t-1,t]+N_{1}}{t}\leq\int_{\Omega}N[0,1]dR_{t,\omega}\leq\frac{N_{t}}{t}+\frac{N[t-1,t]+N_{1}}{t}.
\end{equation}
For the lower bound, for any open ball $B_{\epsilon}(x)$ centered at $x$ with radius $\epsilon>0$,
\begin{align}
P\left(\frac{N_{t}}{t}\in B_{\epsilon}(x)\right)&\geq P\left(\int_{\Omega}N[0,1]dR_{t,\omega}\in B_{\epsilon/2}(x)\right)
\\
&-P\left(\frac{N[t-1,t]}{t}\geq\frac{\epsilon}{4}\right)-P\left(\frac{N_{1}}{t}\geq\frac{\epsilon}{4}\right).\nonumber
\end{align}
For the upper bound, for any closed set $C$ and $C^{\epsilon}=\bigcup_{x\in C}\overline{B_{\epsilon}(x)}$,
\begin{align}
P\left(\frac{N_{t}}{t}\in C\right)&\leq P\left(\int_{\Omega}N[0,1]dR_{t,\omega}\in C^{\epsilon}\right)
\\
&+P\left(\frac{N[t-1,t]}{t}\geq\frac{\epsilon}{4}\right)+P\left(\frac{N_{1}}{t}\geq\frac{\epsilon}{4}\right).\nonumber
\end{align}
Finally, by Lemma \ref{indfinite}, we have the following superexponential estimates
\begin{equation}
\limsup_{t\rightarrow\infty}\frac{1}{t}\log P\left(\frac{N[t-1,t]}{t}\geq\frac{\epsilon}{4}\right)
=\limsup_{t\rightarrow\infty}\frac{1}{t}\log P\left(\frac{N_{1}}{t}\geq\frac{\epsilon}{4}\right)=-\infty.
\end{equation}
Hence, for the lower bound, we have
\begin{equation}
\liminf_{t\rightarrow\infty}\frac{1}{t}\log P\left(\frac{N_{t}}{t}\in B_{\epsilon}(x)\right)\geq -I(x),
\end{equation}
and for the upper bound, we have
\begin{equation}
\limsup_{t\rightarrow\infty}\frac{1}{t}\log P\left(\frac{N_{t}}{t}\in C\right)\leq-\inf_{x\in C^{\epsilon}}I(x),
\end{equation}
which holds for any $\epsilon>0$. Letting $\epsilon\downarrow 0$, we get the desired result.
\end{proof}

\section{Lower Bound}\label{lowerbound}

\begin{lemma}\label{positivity}
For any $\lambda,\hat{\lambda}\geq 0$, $\lambda-\hat{\lambda}+\hat{\lambda}\log(\hat{\lambda}/\lambda)\geq 0$.
\end{lemma}

\begin{proof}
Write $\lambda-\hat{\lambda}+\hat{\lambda}\log(\hat{\lambda}/\lambda)
=\hat{\lambda}\left[(\lambda/\hat{\lambda})-1-\log(\lambda/\hat{\lambda})\right]$. 
Thus, it is sufficient to show that $F(x)=x-1-\log x\geq 0$ for any $x\geq 0$. 
Note that $F(0)=F(\infty)=0$ and $F'(x)=1-\frac{1}{x}<0$ when $0<x<1$ and $F'(x)>0$ 
when $x>1$ and finally $F(1)=0$. Hence $F(x)\geq 0$ for any $x\geq 0$.
\end{proof}

\begin{lemma}\label{finitemean}
Assume $H(Q)<\infty$. Then,
\begin{equation}
\mathbb{E}^{Q}[N[0,1]]\leq C_{1}+C_{2}H(Q),
\end{equation}
where $C_{1},C_{2}>0$ are some constants independent of $Q$.
\end{lemma}

\begin{proof}
If $H(Q)<\infty$, then $h(Q^{\omega^{-}},P^{\omega^{-}})_{\mathcal{F}^{0}_{1}}<\infty$ for a.e. $\omega^{-}$ under $Q$, 
which implies that $Q^{\omega^{-}}\ll P^{\omega^{-}}$ and thus $\hat{A}_{t}\ll A_{t}$, where $\hat{A}_{t}$ and $A_{t}$ 
are the compensators of $N_{t}$ under $Q^{\omega^{-}}$ and $P^{\omega^{-}}$ respectively. 
(For the theory of absolute continuity of point processes and Girsanov formula, 
see for example Lipster and Shiryaev \cite{Lipster} or Daley and Vere-Jones \cite{Daley}.) 
Since $A_{t}=\int_{0}^{t}\lambda(\omega,s)ds$, we have $\hat{A}_{t}=\int_{0}^{t}\hat{\lambda}(\omega,s)ds$ for some $\hat{\lambda}$. 
By the Girsanov formula,
\begin{equation}
H(Q)=\mathbb{E}^{Q}\left[\int_{0}^{1}\lambda-\hat{\lambda}+\log\left(\hat{\lambda}/\lambda\right)\hat{\lambda}ds\right].
\end{equation}
Notice that $\mathbb{E}^{Q}[N[0,1]]=\int\int_{0}^{1}\hat{\lambda}dsdQ$.
\begin{align}
\int\int_{0}^{1}\lambda dsdQ&\leq\epsilon\int\int_{0}^{1}\sum_{\tau<s}h(s-\tau)dsdQ+C_{\epsilon}
\\
&\leq\epsilon\int h(0)N[0,1]dQ+\epsilon\int\sum_{\tau<0}h(-\tau)dQ+C_{\epsilon}\nonumber
\\
&=\epsilon(h(0)+\Vert h\Vert_{L^{1}})\mathbb{E}^{Q}[N[0,1]]+C_{\epsilon}\nonumber
\\
&=\epsilon(h(0)+\Vert h\Vert_{L^{1}})\int\int_{0}^{1}\hat{\lambda}dsdQ+C_{\epsilon}.\nonumber
\end{align}
Therefore, we have
\begin{equation}
\int\int_{0}^{1}\hat{\lambda}\cdot 1_{\hat{\lambda}<K\lambda}dsdQ\leq K\epsilon(h(0)+\Vert h\Vert_{L^{1}})\int\int_{0}^{1}\hat{\lambda}dsdQ+KC_{\epsilon}.
\end{equation}
On the other hand, by Lemma \ref{positivity},
\begin{align}
H(Q)&\geq\int\int_{0}^{1}\left[\lambda-\hat{\lambda}+\hat{\lambda}\log(\hat{\lambda}/\lambda)\right]\cdot 1_{\hat{\lambda}\geq K\lambda}dsdQ
\\
&\geq(\log K-1)\int\int_{0}^{1}\hat{\lambda}\cdot 1_{\hat{\lambda}\geq K\lambda}dsdQ.\nonumber
\end{align}
Thus,
\begin{equation}
\int\int_{0}^{1}\hat{\lambda}dsdQ\leq K\epsilon(h(0)+\Vert h\Vert_{L^{1}})\int\int_{0}^{1}\hat{\lambda}dsdQ+KC_{\epsilon}+\frac{H(Q)}{\log K-1}.
\end{equation}
Choose $K>e$ and $\epsilon<\frac{1}{K(h(0)+\Vert h\Vert_{L^{1}})}$. We get
\begin{equation}
\mathbb{E}^{Q}[N[0,1]]\leq\frac{KC_{\epsilon}}{1-K\epsilon(h(0)+\Vert h\Vert_{L^{1}})}+\frac{H(Q)}{(\log K-1)K\epsilon(h(0)+\Vert h\Vert_{L^{1}})}.
\end{equation}
\end{proof}

\begin{lemma}\label{variational}
We have the following alternative expression for $H(Q)$.
\begin{equation}
H(Q)=\sup_{f(\omega,s)\in\mathcal{B}(\mathcal{F}^{-\infty}_{s})\cap C(\Omega\times\mathbb{R}),0\leq s\leq 1}
\mathbb{E}^{Q}\left[\int_{0}^{1}\lambda(1-e^{f})ds+\int_{0}^{1}fdN_{s}\right].
\end{equation}
\end{lemma}

\begin{proof}
$\mathbb{E}^{Q}[N[0,1]]<\infty$ implies that $\mathbb{E}^{Q^{\omega^{-}}}[N[0,1]]<\infty$ for almost every $\omega^{-}$ under $Q$, 
also $\sum_{\tau\in\omega^{-}}h(-\tau)<\infty$ since $\mathbb{E}^{Q}[\sum_{\tau\in\omega^{-}}h(-\tau)]
=\Vert h\Vert_{L^{1}}\mathbb{E}^{Q}[N[0,1]]<\infty$. Thus,
\begin{align}
\mathbb{E}^{P^{\omega^{-}}}[N[0,1]]&=\mathbb{E}^{P^{\omega^{-}}}\left[\int_{0}^{1}\lambda\left(\sum_{\tau\in\omega[0,s)\cup\omega^{-}}h(s-\tau)\right)ds\right]
\\
&\leq C_{\epsilon}+\epsilon h(0)\mathbb{E}^{P^{\omega^{-}}}[N[0,1]]+\epsilon\sum_{\tau\in\omega^{-}}h(-\tau)<\infty.\nonumber
\end{align}
Pick up $\epsilon<\frac{1}{h(0)}$, we have $\mathbb{E}^{P^{\omega^{-}}}[N[0,1]]<\infty$.

By the theory of absolute continuity of point processes, see for example Chapter 13 of Daley and Vere-Jones \cite{Daley}, 
if $\mathbb{E}^{Q^{\omega^{-}}}[N[0,1]],\mathbb{E}^{P^{\omega^{-}}}[N[0,1]]<\infty$, then, $Q^{\omega^{-}}\ll P^{\omega^{-}}$ 
if and only if $\hat{A}_{t}\ll A_{t}$, where $\hat{A}_{t}$ and $A_{t}=\int_{0}^{t}\lambda(\omega^{-},\omega,s)ds$ 
are the compensators of $N_{t}$ under $Q^{\omega^{-}}$ and $P^{\omega^{-}}$ respectively. 
If that's the case, we can write $\hat{A}_{t}=\int_{0}^{t}\hat{\lambda}(\omega^{-},\omega,s)ds$ for some $\hat{\lambda}$ and there is Girsanov formula
\begin{equation}
\log\frac{dQ^{\omega^{-}}}{dP^{\omega^{-}}}\bigg|_{\mathcal{F}^{0}_{1}}=\int_{0}^{1}(\lambda-\hat{\lambda})ds
+\int_{0}^{1}\log\left(\hat{\lambda}/\lambda\right)dN_{s},
\end{equation}
which implies that
\begin{equation}
H(Q)=\mathbb{E}^{Q}\left[\int_{0}^{1}\lambda-\hat{\lambda}+\log\left(\hat{\lambda}/\lambda\right)\hat{\lambda}ds\right].
\end{equation}
For any $f$, $\hat{\lambda}f+(1-e^{f})\lambda\leq\hat{\lambda}\log(\hat{\lambda}/\lambda)+\lambda-\hat{\lambda}$ 
and equality is achieved when $f=\log(\hat{\lambda}/\lambda)$. Thus, clearly, we have
\begin{equation}
\sup_{f(\omega,s)\in\mathcal{B}(\mathcal{F}^{-\infty}_{s})\cap C(\Omega\times\mathbb{R}),0\leq s\leq 1}\mathbb{E}^{Q}
\left[\int_{0}^{1}\lambda(1-e^{f})ds+\int_{0}^{1}fdN_{s}\right]\leq H(Q).
\end{equation}
On the other hand, we can always find a sequence $f_{n}$ convergent to $\log(\hat{\lambda}/\lambda)$ and by Fatou's lemma, we get the other direction.

Now, assume that we do not have $Q^{\omega^{-}}\ll P^{\omega^{-}}$ for a.e. $\omega^{-}$ under $Q$. That implies that $H(Q)=\infty$.
We want to show that
\begin{equation}
\sup_{f(\omega,s)\in\mathcal{B}(\mathcal{F}^{-\infty}_{s})\cap C(\Omega\times\mathbb{R}),0\leq s\leq 1}
\mathbb{E}^{Q}\left[\int_{0}^{1}\lambda(1-e^{f})ds+\int_{0}^{1}fdN_{s}\right]=\infty.
\end{equation}
Let us assume that
\begin{equation}
\sup_{f(\omega,s)\in\mathcal{B}(\mathcal{F}^{-\infty}_{s})\cap C(\Omega\times\mathbb{R}),0\leq s\leq 1}
\mathbb{E}^{Q}\left[\int_{0}^{1}\lambda(1-e^{f})ds+\int_{0}^{1}fdN_{s}\right]<\infty.
\end{equation}
We want to prove that $H(Q)<\infty$.

Let $P^{\omega^{-}}_{\epsilon}$ be the point process on $[0,1]$ with compensator $A_{t}+\epsilon\hat{A}_{t}$. 
Clearly $\hat{A}_{t}\ll A_{t}+\epsilon\hat{A}_{t}$ and $Q^{\omega^{-}}\ll P^{\omega^{-}}_{\epsilon}$.

For any $f$,
\begin{align}
&\mathbb{E}^{Q}\left[\int_{0}^{1}(1-e^{f})d(A_{s}+\epsilon\hat{A}_{s})+fd\hat{A}_{s}\right]
\\
&=\mathbb{E}^{Q}\left[\int_{0}^{1}(1-e^{f})\chi_{f<0}d(A_{s}+\epsilon\hat{A}_{s})+f\chi_{f<0}d\hat{A}_{s}\right]\nonumber
\\
&+\mathbb{E}^{Q}\left[\int_{0}^{1}(1-e^{f})\chi_{f\geq 0}d(A_{s}+\epsilon\hat{A}_{s})+f\chi_{f\geq 0}d\hat{A}_{s}\right]\nonumber
\\
&\leq\mathbb{E}^{Q}\left[\int_{0}^{1}d(A_{s}+\epsilon\hat{A}_{s})\right]+\mathbb{E}^{Q}\left[\int_{0}^{1}(1-e^{f})
\chi_{f\geq 0}dA_{s}+f\chi_{f\geq 0}d\hat{A}_{s}\right]\nonumber
\\
&=\mathbb{E}^{Q}\left[\int_{0}^{1}d(A_{s}+\epsilon\hat{A}_{s})\right]
+\mathbb{E}^{Q}\left[\int_{0}^{1}(1-e^{f\chi_{f\geq 0}})dA_{s}+f\chi_{f\geq 0}d\hat{A}_{s}\right]\nonumber
\\
&\leq C_{\delta}+\delta(h(0)+\Vert h\Vert_{L^{1}})\mathbb{E}^{Q}[N[0,1]]\nonumber
\\
&+\sup_{f(\omega,s)\in\mathcal{B}(\mathcal{F}^{-\infty}_{s})\cap C(\Omega\times\mathbb{R}),0\leq s\leq 1}
\mathbb{E}^{Q}\left[\int_{0}^{1}\lambda(1-e^{f})ds+\int_{0}^{1}fdN_{s}\right]<\infty.\nonumber
\end{align}
Therefore, 
\begin{align}
\infty&>\liminf_{\epsilon\downarrow 0}\sup_{f(\omega,s)\in\mathcal{B}(\mathcal{F}^{-\infty}_{s})
\cap C(\Omega\times\mathbb{R}),0\leq s\leq 1}\mathbb{E}^{Q}\left[\int_{0}^{1}(1-e^{f})d(A_{s}+\epsilon\hat{A}_{s})+fd\hat{A}_{s}\right]
\\
&=\liminf_{\epsilon\downarrow 0}\sup_{f(\omega,s)\in\mathcal{B}(\mathcal{F}^{-\infty}_{s})
\cap C(\Omega\times\mathbb{R}),0\leq s\leq 1}\nonumber
\\
&\qquad\qquad\qquad\qquad\qquad\mathbb{E}^{Q}\left[\int_{0}^{1}\left(1-e^{f}
+f\cdot\frac{d\hat{A}_{s}}{d(A_{s}+\epsilon\hat{A}_{s})}\right)d(A_{s}+\epsilon\hat{A}_{s})\right]\nonumber
\\
&=\liminf_{\epsilon\downarrow 0}\mathbb{E}^{Q}[h(Q^{\omega^{-}},P^{\omega^{-}}_{\epsilon})_{\mathcal{F}^{0}_{1}}]\nonumber
\\
&=\mathbb{E}^{Q}[h(Q^{\omega^{-}},P^{\omega^{-}})_{\mathcal{F}^{0}_{1}}]=H(Q),\nonumber
\end{align}
by lower semicontinuity of the relative entropy $h(\cdot,\cdot)$, Fatou's lemma, and the fact that 
$P^{\omega^{-}}_{\epsilon}\rightarrow P^{\omega^{-}}$ weakly as $\epsilon\downarrow 0$. Hence $H(Q)<\infty$.
\end{proof}

\begin{lemma}\label{lscconvex}
$H(Q)$ is lower semicontinuous and convex in $Q$.
\end{lemma}

\begin{proof}
By Lemma \ref{variational}, we can rewrite $H(Q)$ as
\begin{align}
H(Q)&=\sup_{f(\omega,s)\in\mathcal{B}(\mathcal{F}^{-\infty}_{s})\cap C(\Omega\times\mathbb{R}),0\leq s\leq 1}
\mathbb{E}^{Q}\left[\int_{0}^{1}\lambda(1-e^{f})+\hat{\lambda}fds\right]
\\
&=\sup_{f(\omega,s)\in\mathcal{B}(\mathcal{F}^{-\infty}_{s})\cap C(\Omega\times\mathbb{R}),0\leq s\leq 1}
\mathbb{E}^{Q}\left[\int_{0}^{1}\lambda(1-e^{f})ds+\int_{0}^{1}fdN_{s}\right].\nonumber
\end{align}
If $Q_{n}\rightarrow Q$, then $\mathbb{E}^{Q_{n}}[N[0,1]]\rightarrow\mathbb{E}^{Q}[N[0,1]]$ and $Q_{n}\rightarrow Q$ weakly. 
Since $f(\omega,s)\in C(\Omega\times\mathbb{R})\cap\mathcal{B}(\mathcal{F}^{-\infty}_{s})$, $\int_{0}^{1}f(\omega,s)dN_{s}$ is continuous 
on $\Omega$ and since $f$ is uniformly bounded, $\int_{0}^{1}f(\omega,s)dN_{s}\leq\Vert f\Vert_{L^{\infty}}N[0,1]$. Hence,
\begin{equation}
\mathbb{E}^{Q_{n}}\left[\int_{0}^{1}f(\omega,s)dN_{s}\right]\rightarrow\mathbb{E}^{Q}\left[\int_{0}^{1}f(\omega,s)dN_{s}\right].
\end{equation}
Let $\lambda^{M}=\lambda\left(\sum_{\tau<s}h^{M}(s-\tau)\right)$, where $h^{M}(s)=h(s)\chi_{s\leq M}$. 
Then, $\lambda^{M}(\omega,s)\in C(\Omega\times\mathbb{R})$ and thus $\int_{0}^{1}\lambda^{M}(1-e^{f(\omega,s)})ds\in C(\Omega)$. 
Also, $\int_{0}^{1}\lambda^{M}(1-e^{f(\omega,s)})ds\leq K(1+e^{\Vert f\Vert_{L^{\infty}}})N[-M,1]$, where $K>0$ is some constant. 
Therefore,
\begin{equation}
\mathbb{E}^{Q_{n}}\left[\int_{0}^{1}\lambda^{M}(1-e^{f(\omega,s)})ds\right]
\rightarrow\mathbb{E}^{Q}\left[\int_{0}^{1}\lambda^{M}(1-e^{f(\omega,s)})ds\right]
\end{equation} 
as $n\rightarrow\infty$. Next, notice that
\begin{align}
&\left|\mathbb{E}^{Q}\left[\int_{0}^{1}\lambda^{M}(1-e^{f(\omega,s)})ds\right]-\mathbb{E}^{Q}\left[\int_{0}^{1}\lambda(1-e^{f(\omega,s)})ds\right]\right|
\\
&\leq\mathbb{E}^{Q}(1+e^{\Vert f\Vert_{L^{\infty}}})\alpha\mathbb{E}^{Q}[N[0,1]]\int_{M}^{\infty}h(s)ds\rightarrow 0\nonumber
\end{align}
as $M\rightarrow\infty$. Similarly, we have
\begin{equation}
\limsup_{M\rightarrow\infty}\limsup_{n\rightarrow\infty}\left|\mathbb{E}^{Q_{n}}
\left[\int_{0}^{1}\lambda^{M}(1-e^{f(\omega,s)})ds\right]-\mathbb{E}^{Q_{n}}\left[\int_{0}^{1}\lambda(1-e^{f(\omega,s)})ds\right]\right|=0.
\end{equation}
Hence,
\begin{equation}
\mathbb{E}^{Q_{n}}\left[\int_{0}^{1}\lambda(\omega,s)(1-e^{f(\omega,s)})ds\right]
\rightarrow\mathbb{E}^{Q}\left[\int_{0}^{1}\lambda(\omega,s)(1-e^{f(\omega,s)})ds\right].
\end{equation}
The supremum is taken over a linear functional of $Q$, which is continuous in $Q$, 
therefore the supremum over these linear functionals will be lower semicontinuous. 
Similarly, since in the variational formula expression 
of $H(Q)$ in Lemma \ref{variational}, the supremum is taken over a linear functional of $Q$, $H(Q)$ is convex in $Q$.
\end{proof}

\begin{lemma}
$H(Q)$ is linear in $Q$.
\end{lemma}

\begin{proof}
It is in general true that the process-level entropy function $H(Q)$ is linear in $Q$. Following the arguments in Donsker and Varadhan \cite{Donsker},
there exists a subset $\Omega_{0}\subset\Omega$ which is $\mathcal{F}^{-\infty}_{0}$ measurable and a $\mathcal{F}^{-\infty}_{0}$ 
measurable map $\hat{Q}:\Omega_{0}\rightarrow\mathcal{M}_{E}(\Omega)$ such that $Q(\Omega_{0})=1$ for all $Q\in\mathcal{M}_{S}(\Omega)$
and $Q(\omega:\hat{Q}=Q)=1$ for all $Q\in\mathcal{M}_{E}(\Omega)$. Therefore, there exists a universal version,
say $\hat{Q}^{\omega^{-}}$ independent of $Q$ such that $\int\hat{Q}^{\omega^{-}}Q(d\omega^{-})=Q$. Since that is
true for all $Q\in\mathcal{M}_{E}(\Omega)$, it also holds for $Q\in\mathcal{M}_{S}(\Omega)$. Hence, 
\begin{equation}
H(Q)=\mathbb{E}^{Q}\left[h(Q^{\omega^{-}},P^{\omega^{-}})_{\mathcal{F}^{0}_{1}}\right]
=\mathbb{E}^{Q}\left[h(\hat{Q}^{\omega^{-}},P^{\omega^{-}})_{\mathcal{F}^{0}_{1}}\right],
\end{equation}
i.e. $H(Q)$ is linear in $Q$.
\end{proof}

In our paper, we are proving the large deviation principle
for Hawkes processes started with empty history, i.e. with probability measure $P^{\emptyset}$. But when
time elapses by $t$, the Hawkes process generates points and that create a new history up to time $t$. We need
to understand how the history that was created would affect the future. What we want to prove is some uniform
estimates that if given the past history which is well controlled, then the new history will also be well controlled.
This is essentially what the following Lemma \ref{mainlower} says. Consider the configuration of points
starting from time $0$ up to time $t$. We shift it by $t$ and denote that by $w_{t}$ such that $w_{t}\in\Omega^{-}$, where $\Omega^{-}$ is
$\Omega$ restricted to $\mathbb{R}^{-}$.
These notations will be used in Lemma \ref{mainlower}.

\begin{remark}
At the very beginning of the paper, we defined $\omega_{t}$. It should not be confused with $w_{t}$ in this section.
\end{remark}

\begin{lemma}\label{mainlower}
For any $Q\in\mathcal{M}_{E}(\Omega)$ such that $H(Q)<\infty$ and any open neighborhood $N$ of $Q$, 
there exists some $K^{-}_{\ell}$ such that $\emptyset\in K^{-}_{\ell}$ and $Q(K^{-}_{\ell})\rightarrow 1$ as $\ell\rightarrow\infty$ and
\begin{equation}
\liminf_{t\rightarrow\infty}\frac{1}{t}\inf_{w_{0}\in K^{-}_{\ell}}\log P^{w_{0}}(R_{t,\omega}\in N,w_{t}\in K^{-}_{\ell})\geq-H(Q).
\end{equation}
\end{lemma}

\begin{proof}
Let us abuse the notations a bit by defining
\begin{equation}
\lambda(\omega^{-})=\lambda\left(\sum_{\tau\in\omega^{-},\tau\in\omega[0,s)}h(s-\tau)\right).
\end{equation}

For any $t>0$, since $\lambda(\cdot)\geq c>0$ and $\lambda(\cdot)$ is Lipschitz with constant $\alpha$, we have
\begin{align}
\log\frac{dP^{\omega^{-}}}{dP^{w_{0}}}\bigg|_{\mathcal{F}^{0}_{t}}
&=\int_{0}^{t}\lambda(w_{0})-\lambda(\omega^{-})ds+\int_{0}^{t}\log\left(\frac{\lambda(\omega^{-})}{\lambda(w_{0})}\right)dN_{s}
\\
&\leq\int_{0}^{t}|\lambda(w_{0})-\lambda(\omega^{-})|ds
+\int_{0}^{t}\log\left(1+\frac{|\lambda(w_{0})-\lambda(\omega^{-})|}{\lambda(w_{0})}\right)dN_{s}\nonumber
\\
&\leq\int_{0}^{t}\alpha\sum_{\tau\in\omega^{-}\cup w_{0}}h(s-\tau)ds+\int_{0}^{t}\frac{\alpha}{c}\sum_{\tau\in\omega^{-}\cup w_{0}}h(s-\tau)dN_{s}.\nonumber
\end{align}
Define
\begin{equation}
K^{-}_{\ell}=\left\{\omega:N[-t,0](\omega)\leq\ell(1+t), \forall t>0\right\}.
\end{equation}
By the maximal ergodic theorem, letting $[t]$ denote the largest integer less or equal to $t$, we have
\begin{align}
Q((K^{-}_{\ell})^{c})&=Q\left(\sup_{t>0}\frac{N[-t,0]}{t+1}>\ell\right)
\\
&\leq Q\left(\sup_{t>0}\frac{N[-([t]+1),0]}{[t]+1}>\ell\right)\nonumber
\\
&=Q\left(\sup_{n\geq 1, n\in\mathbb{N}}\frac{N[-n,0]}{n}>\ell\right)\nonumber
\\
&\leq\frac{\mathbb{E}^{Q}[N[0,1]]}{\ell}\rightarrow 0
\end{align}
as $\ell\rightarrow\infty$. Thus $Q(K^{-}_{\ell})\rightarrow 1$ as $\ell\rightarrow\infty$.

Fix any $s>0$ and $\omega^{-}\in K^{-}_{\ell}$. Since $h$ is decreasing, $h'\leq 0$, integration by parts shows that
\begin{align}
\sum_{\tau\in\omega^{-}}h(s-\tau)&=\int_{0}^{\infty}h(s+\sigma)dN[-\sigma,0]
\\
&=-\int_{0}^{\infty}N[-\sigma,0]h'(s+\sigma)d\sigma\nonumber
\\
&\leq-\int_{0}^{\infty}\ell(1+\sigma)h'(s+\sigma)d\sigma\nonumber
\\
&=\ell h(s)+\ell\int_{0}^{\infty}h(s+\sigma)d\sigma\nonumber
\\
&=\ell h(s)+\ell H(s),\nonumber
\end{align}
where $H(t)=\int_{t}^{\infty}h(s)ds$.

Therefore, uniformly for $\omega_{-},w_{0}\in K^{-}_{\ell}$,
\begin{equation}
\int_{0}^{t}\alpha\sum_{\tau\in\omega^{-}\cup w_{0}}h(s-\tau)ds\leq 2\ell\alpha\Vert h\Vert_{L^{1}}+2\ell\alpha u(t),
\end{equation}
where $u(t)=\int_{0}^{t}H(s)ds$ and
\begin{equation}
\int_{0}^{t}\frac{\alpha}{c}\sum_{\tau\in\omega^{-}\cup w_{0}}h(s-\tau)dN_{s}\leq\frac{2\ell\alpha}{c}\int_{0}^{t}(h(s)+H(s))dN_{s}.
\end{equation}
Define
\begin{equation}
K^{+}_{\ell,t}=\left\{\omega:\frac{2\ell\alpha}{c}\int_{0}^{t}(h(s)+H(s))dN_{s}\leq\ell^{2}(\Vert h\Vert_{L^{1}}+u(t))\right\}.
\end{equation}
Then, uniformly in $t>0$,
\begin{equation}
Q((K^{+}_{\ell,t})^{c})\leq\frac{2\alpha\mathbb{E}^{Q}[N[0,1]]}{c\cdot\ell}\rightarrow 0,
\end{equation}
as $\ell\rightarrow\infty$. Thus $\inf_{t>0}Q(K^{+}_{\ell,t})\rightarrow 1$ as $\ell\rightarrow\infty$.

Hence, uniformly for $\omega_{-},w_{0}\in K^{-}_{\ell}$ and $\omega\in K^{+}_{\ell,t}$,
\begin{align}
\log\frac{dP^{\omega^{-}}}{dP^{w_{0}}}\bigg|_{\mathcal{F}^{0}_{t}}&\leq 2\ell\alpha\Vert h\Vert_{L^{1}}+2\ell\alpha u(t)+\ell^{2}(\Vert h\Vert_{L^{1}}+u(t))
\\
&=C_{1}(\ell)+C_{2}(\ell)u(t),
\end{align}
where $C_{1}(\ell)=2\ell\alpha\Vert h\Vert_{L^{1}}+\ell^{2}\Vert h\Vert_{L^{1}}$ and $C_{2}(\ell)=2\ell\alpha+\ell^{2}$.

Observe that
\begin{equation}
\limsup_{t\rightarrow\infty}\frac{u(t)}{t}=\limsup_{t\rightarrow\infty}\frac{1}{t}\int_{0}^{t}H(s)ds=0.
\end{equation}

Let $D_{t}=\{R_{t,\omega}\in N,w_{t}\in K^{-}_{\ell}\}$. 

Uniformly for $w_{0}\in K^{-}_{\ell,t}$,
\begin{align}
&P^{w_{0}}(D_{t})
\\
&\geq e^{-t(H(Q)+\epsilon)-C_{1}(\ell)-C_{2}(\ell)u(t)}\nonumber
\\
&\cdot Q\left[D_{t}\cap\left\{\frac{1}{t}\log\frac{dP^{\omega^{-}}}{dQ^{\omega^{-}}}\bigg|_{\mathcal{F}^{0}_{t}}
\leq H(Q)+\epsilon\right\}\cap\left\{\log\frac{dP^{\omega^{-}}}{dP^{w_{0}}}\bigg|_{\mathcal{F}^{0}_{t}}
\leq C_{1}(\ell)+C_{2}(\ell)u(t)\right\}\right]\nonumber
\\
&\geq e^{-t(H(Q)+\epsilon)-C_{1}(\ell)-C_{2}(\ell)u(t)}\nonumber
\\
&\qquad\qquad\qquad\qquad\cdot 
Q\left[D_{t}\cap\left\{\frac{1}{t}\log\frac{dP^{\omega^{-}}}{dQ^{\omega^{-}}}\bigg|_{\mathcal{F}^{0}_{t}}
\leq H(Q)+\epsilon\right\}\cap\{K^{+}_{\ell,t}\cap K^{-}_{\ell}\}\right].\nonumber
\end{align}

Since $Q\in\mathcal{M}_{E}(\Omega)$, by ergodic theorem,
\begin{equation}
\lim_{t\rightarrow\infty}Q(R_{t,\omega}\in N)=1,
\end{equation}
and since $\psi(\omega,t)=\log\frac{dQ^{\omega}}{dP^{\omega}}\big|_{\mathcal{F}^{0}_{t}}$ satisfies
\begin{equation}
\psi(\omega,t+s)=\psi(\omega,t)+\psi(\theta_{t}\omega,s),\quad\mathbb{E}^{Q}[\psi(\omega,t)]=tH(Q),
\end{equation}
for almost every $\omega^{-}$ under $Q$,
\begin{equation}
\lim_{t\rightarrow\infty}\frac{1}{t}\log\frac{dP^{\omega^{-}}}{dQ^{\omega^{-}}}\bigg|_{\mathcal{F}^{0}_{t}}=H(Q).
\end{equation}
Since $Q$ is stationary, $Q(w_{t}\in K^{-}_{\ell})\geq Q(K^{-}_{\ell})\rightarrow 1$ as $\ell\rightarrow\infty$. 
Also, $Q(K^{+}_{\ell,t})\geq\inf_{t>0}Q(K^{+}_{\ell,t})\rightarrow 1$ as $\ell\rightarrow\infty$. 
Remember that $\limsup_{t\rightarrow\infty}\frac{u(t)}{t}=0$. By choosing $\ell$ big enough, we conclude that
\begin{equation}
\liminf_{t\rightarrow\infty}\frac{1}{t}\inf_{w_{0}\in K^{-}_{\ell}}\log P^{w_{0}}(R_{t,\omega}\in N,w_{t}\in K^{-}_{\ell})\geq-H(Q)-\epsilon.
\end{equation}
Since it holds for any $\epsilon>0$, we get the desired result.
\end{proof}

\begin{theorem}[Lower Bound]
For any open set $G$, 
\begin{equation}
\liminf_{t\rightarrow\infty}\frac{1}{t}\log P(R_{t,\omega}\in G)\geq-\inf_{Q\in G}H(Q).
\end{equation}
\end{theorem}

\begin{proof}
It is sufficent to prove that for any $Q\in\mathcal{M}_{S}(\Omega)$, $H(Q)<\infty$, for any neighborhood $N$ of $Q$, 
$\liminf_{t\rightarrow\infty}\frac{1}{t}\log P(R_{t,\omega}\in N)\geq -H(Q)$. 
Since for every invariant measure $P\in\mathcal{M}_{S}$, 
there exists a probability measure $\mu_{P}$ on the space $\mathcal{M}_{E}$ of ergodic measures 
such that $P=\int_{\mathcal{M}_{E}}Q\mu_{P}(dQ)$, for any $Q\in\mathcal{M}_{S}(\Omega)$ such that $H(Q)<\infty$, 
without loss of generality, we can assume that $Q=\sum_{j=1}^{\ell}\alpha_{j}Q_{j}$, 
where $\alpha_{j}\geq 0$, $1\leq j\leq\ell$ and $\sum_{j=1}^{\ell}\alpha_{j}=1$. 
By linearity of $H(\cdot)$, $H(Q)=\sum_{j=1}^{\ell}\alpha_{j}H(Q_{j})$. 
Divide the interval $[0,t]$ into subintervals of length $\alpha_{j}t$,
let $t_{j}$, $1\leq j\leq\ell$ be the right hand endpoints of these subintervals, and let $t_{0}=0$. 
For each $Q_{j}$, take $K^{-}_{M}$ as in Lemma \ref{mainlower}. 
We have $\min_{1\leq j\leq\ell}Q_{j}(K^{-}_{M})\rightarrow 1$, as $M\rightarrow\infty$. 
Choose neighborhoods $N_{j}$ of $Q_{j}$, $1\leq j\leq\ell$ such that $\bigcup_{j=1}^{\ell}\alpha_{j}N_{j}\subseteq N$.  We have
\begin{align}
P^{\emptyset}(R_{t,\omega}\in N)&\geq P^{\emptyset}(R_{t_{1},\omega}\in N_{1},w_{t_{1}}\in K^{-}_{M})
\\
&\cdot\prod_{j=2}^{\ell}\inf_{w_{0}\in K^{-}_{t_{j-1}-t_{j-2}}}P^{w_{0}}(R_{t_{j}-t_{j-1},\omega}\in N_{j},w_{t_{j}-t_{j-1}}\in K^{-}_{M}).\nonumber
\end{align}
Now, applying Lemma \ref{mainlower} and the linearity of $H(\cdot)$,
\begin{equation}
\liminf_{t\rightarrow\infty}\frac{1}{t}\log P^{\emptyset}(R_{t,\omega}\in N)\geq-\sum_{j=1}^{\ell}\alpha_{j}H(Q_{j})=-H(Q).
\end{equation}
\end{proof}

\section{Upper Bound}\label{upperbound}

\begin{remark}
By following the same argument as in Donsker and Varadhan \cite{Donsker}, if $\omega^{-}\mapsto P^{\omega^{-}}$ is weakly continuous, then
\begin{equation}
\limsup_{t\rightarrow\infty}\frac{1}{t}\log P(R_{t,\omega}\in A)\leq-\inf_{Q\in A}H(Q),
\end{equation}
for any compact $A$. If the Hawkes process has finite range of memory, i.e. $h(\cdot)$ has compact support and is continuous, 
then, for any $a<b$, if $\omega^{-}_{n}\rightarrow\omega^{-}$, we have
\begin{align}
&\left|\int_{a}^{b}\lambda(\omega^{-}_{n},\omega,s)ds-\int_{a}^{b}\lambda(\omega^{-},\omega,s)ds\right|
\\
&\leq\alpha\int_{a}^{b}\left|\sum_{\tau\in\omega^{-}_{n}}h(s-\tau)-\sum_{\tau\in\omega^{-}}h(s-\tau)\right|ds\rightarrow\infty,\nonumber
\end{align}
as $n\rightarrow\infty$, which implies that $P^{\omega^{-}_{n}}\rightarrow P^{\omega^{-}}$.

If the Hawkes process does not have finite range of memory, then, we should use the specific features of the Hawkes process to obtain the upper bound.
\end{remark}

Before we proceed, let us prove an easy but very useful lemma that we will use repeatedly in the proofs of the estimates in this paper.

\begin{lemma}
Let $f(\omega,s)$ be $\mathcal{F}^{-\infty}_{s}$ progressively measurable and predictable. Then,
\begin{equation}
\mathbb{E}\left[e^{\int_{0}^{t}f(\omega,s)dN_{s}}\right]\leq\mathbb{E}\left[e^{\int_{0}^{t}(e^{2f(\omega,s)}-1)\lambda(\omega,s)ds}\right]^{1/2}.
\end{equation} 
\end{lemma}

\begin{proof}
Since $\exp\left\{\int_{0}^{t}2f(\omega,s)dN_{s}-\int_{0}^{t}(e^{2f(\omega,s)}-1)\lambda(\omega,s)ds\right\}$ is a martingale, 
by Cauchy-Schwarz inequality,
\begin{align}
\mathbb{E}\left[e^{\int_{0}^{t}f(\omega,s)dN_{s}}\right]
&=\mathbb{E}\left[e^{\frac{1}{2}\int_{0}^{t}2f(\omega,s)dN_{s}-\frac{1}{2}\int_{0}^{t}(e^{2f(\omega,s)}-1)\lambda(\omega,s)ds
+\frac{1}{2}\int_{0}^{t}(e^{2f(\omega,s)}-1)\lambda(\omega,s)ds}\right]
\\
&\leq\mathbb{E}\left[e^{\int_{0}^{t}(e^{2f(\omega,s)}-1)\lambda(\omega,s)ds}\right]^{1/2}.\nonumber
\end{align}
\end{proof}

Let us define $\mathcal{C}_{T}$ as
\begin{align}
\mathcal{C}_{T}&=\bigg\{F(\omega):=\int_{0}^{T}f(\omega,s)dN_{s}-\int_{0}^{T}(e^{f(\omega,s)}-1)\lambda(\omega,s)ds, 
\\
&\qquad\qquad\qquad\qquad\qquad\qquad\qquad f(\omega,s)\in\mathcal{B}(\mathcal{F}^{0}_{s})\cap C(\Omega\times\mathbb{R})\bigg\}.\nonumber
\end{align}
Here $\lambda(\omega,s)$ is $\mathcal{F}^{-\infty}_{s}$ progressively measurable and predictable, and 
$f(\omega,s)\in\mathcal{B}(\mathcal{F}^{0}_{s})\cap C(\Omega\times\mathbb{R})$ means that 
$f$ is $\mathcal{F}^{0}_{s}$ progressively measurable, predictable and also bounded and continuous.

\begin{lemma}\label{expectation}
For any $T>0$ and $F\in\mathcal{C}_{T}$, we have, for any $t>0$,
\begin{equation}
\mathbb{E}^{P^{\emptyset}}\left[e^{\frac{1}{T}\int_{0}^{t}F(\theta_{s}\omega)ds}\right]\leq 1.
\end{equation}
\end{lemma}

\begin{proof}
For any $t>0$, writing $\psi(s)=\sum_{k:s+kT\leq t}F(\theta_{s+kT}\omega)$,
\begin{align}
\mathbb{E}^{P^{\emptyset}}\left[e^{\frac{1}{T}\int_{0}^{t}F(\theta_{s}\omega)ds}\right]
&=\mathbb{E}^{P^{\emptyset}}\left[e^{\frac{1}{T}\int_{0}^{T}\psi(s)ds}\right]
\\
&\leq\frac{1}{T}\int_{0}^{T}\mathbb{E}^{P^{\emptyset}}\left[e^{\psi(s)}\right]ds=1,\nonumber
\end{align}
by Jensen's inequality and the fact that $\mathbb{E}^{P^{\emptyset}}\left[e^{\psi(s)}\right]=1$ 
by iteratively conditioning since $\mathbb{E}^{P^{\omega^{-}}}\left[e^{F(\omega)}\right]=1$ for any $\omega^{-}$.
\end{proof}

\begin{remark}
Under $P^{\emptyset}$, the $\mathcal{F}^{-\infty}_{s}$ progressively measurable rate function $\lambda$ 
is well defined since it only creates a history between time $0$ and time $t$. 
Similary, in the proof in Lemma \ref{expectation}, $\mathbb{E}^{P^{\omega^{-}}}\left[e^{F(\omega)}\right]=1$ for any $\omega^{-}$ 
should be interpreted as the expectation is $1$ given any history created between time $0$ and $t$, which is well defined.
\end{remark}

Next, we need to compare $\frac{1}{T}\int_{0}^{t}F(\theta_{s}\omega_{t})ds$ and $\frac{1}{T}\int_{0}^{t}F(\theta_{s}\omega)ds$.

\begin{lemma}\label{difference}
For any $q>0$, $T>0$ and $F\in\mathcal{C}_{T}$,
\begin{equation}
\limsup_{t\rightarrow\infty}\frac{1}{t}\log\mathbb{E}^{P^{\emptyset}}
\left[\exp\left\{q\left|\frac{1}{T}\int_{0}^{t}F(\theta_{s}\omega_{t})ds-\frac{1}{T}\int_{0}^{t}F(\theta_{s}\omega)ds\right|\right\}\right]=0.
\end{equation}
\end{lemma}

\begin{proof}
\begin{align}
&\left|\frac{1}{T}\int_{0}^{t}F(\theta_{s}\omega_{t})ds-\frac{1}{T}\int_{0}^{t}F(\theta_{s}\omega)ds\right|
\\
&\leq\left|\frac{1}{T}\int_{0}^{t}\int_{0}^{T}f(u,\theta_{s}\omega)dN_{u}ds-\frac{1}{T}\int_{0}^{t}
\int_{0}^{T}f(u,\theta_{s}\omega_{t})dN_{u}ds\right|\nonumber
\\
&\qquad\qquad\qquad+\bigg|\frac{1}{T}\int_{0}^{t}\int_{0}^{T}(e^{f(u,\theta_{s}\omega)}-1)\lambda(\theta_{s}\omega,u)duds\nonumber
\\
&\qquad\qquad\qquad\qquad\qquad\qquad
-\frac{1}{T}\int_{0}^{t}\int_{0}^{T}(e^{f(u,\theta_{s}\omega_{t})}-1)\lambda(\theta_{s}\omega_{t},u)duds\bigg|.\nonumber
\end{align}
It is easy to see that $\int_{0}^{T}f(u,\theta_{s}\omega)dN_{u}ds$ is $\mathcal{F}^{s}_{s+T}$-measurable and that
\begin{equation}
\int_{0}^{T}f(u,\theta_{s}\omega)dN_{u}ds=\int_{0}^{T}f(u,\theta_{s}\omega_{t})dN_{u}ds
\end{equation}
for any $0\leq s\leq t-T$. Hence,
\begin{align}
&\left|\frac{1}{T}\int_{0}^{t}\int_{0}^{T}f(u,\theta_{s}\omega)dN_{u}ds-\frac{1}{T}\int_{0}^{t}\int_{0}^{T}f(u,\theta_{s}\omega_{t})dN_{u}ds\right|
\\
&\leq\frac{1}{T}\int_{t-T}^{t}\int_{0}^{T}|f(u,\theta_{s}\omega)|dN_{u}ds+\frac{1}{T}\int_{t-T}^{t}\int_{0}^{T}|f(u,\theta_{s}\omega_{t})|dN_{u}ds\nonumber
\\
&\leq\frac{\Vert f\Vert_{L^{\infty}}}{T}\int_{t-T}^{t}N[s,s+T](\omega)ds+\frac{\Vert f\Vert_{L^{\infty}}}{T}\int_{t-T}^{t}N[s,s+T](\omega_{t})ds\nonumber
\\
&\leq\frac{\Vert f\Vert_{L^{\infty}}}{T}\left[N[t-T,t+T](\omega)+N[t-T,t+T](\omega_{t})\right]\nonumber
\\
&=\frac{\Vert f\Vert_{L^{\infty}}}{T}\left[N[t-T,t+T](\omega)+N[t-T,T](\omega)+N[0,T](\omega)\right].\nonumber
\end{align}
By H\"{o}lder's inequality and Lemma \ref{indfinite}, we have
\begin{align}
&\limsup_{t\rightarrow\infty}\frac{1}{t}\log\mathbb{E}^{P^{\emptyset}}
\left[e^{\left|\frac{1}{T}\int_{0}^{t}\int_{0}^{T}f(u,\theta_{s}\omega)dN_{u}ds
-\frac{1}{T}\int_{0}^{t}\int_{0}^{T}f(u,\theta_{s}\omega_{t})dN_{u}ds\right|}\right]
\\
&\leq\limsup_{t\rightarrow\infty}\frac{1}{t}\log\mathbb{E}^{P^{\emptyset}}
\left[e^{\frac{\Vert f\Vert_{L^{\infty}}}{T}\left[N[t-T,t+T](\omega)+N[t-T,T](\omega)+N[0,T](\omega)\right]}\right]=0.\nonumber
\end{align}
Furthermore,
\begin{align}
&\left|\frac{1}{T}\int_{0}^{t}\int_{0}^{T}(e^{f(u,\theta_{s}\omega)}-1)\lambda(\theta_{s}\omega,u)duds
-\frac{1}{T}\int_{0}^{t}\int_{0}^{T}(e^{f(u,\theta_{s}\omega_{t})}-1)\lambda(\theta_{s}\omega_{t},u)duds\right|
\\
&\leq\frac{1}{T}\int_{0}^{t}\int_{0}^{T}\left|e^{f(u,\theta_{s}\omega)}-e^{f(u,\theta_{s}\omega_{t})}\right|\lambda(\theta_{s}\omega,u)duds\nonumber
\\
&+\frac{1}{T}\int_{0}^{t}\int_{0}^{T}(e^{f(u,\theta_{s}\omega_{t})}-1)\left|\lambda(\theta_{s}\omega_{t},u)-\lambda(\theta_{s}\omega,u)\right|duds.\nonumber
\end{align}
For the first term
\begin{align}
&\frac{1}{T}\int_{0}^{t}\int_{0}^{T}\left|e^{f(u,\theta_{s}\omega)}-e^{f(u,\theta_{s}\omega_{t})}\right|\lambda(\theta_{s}\omega,u)duds
\\
&=\frac{1}{T}\int_{t-T}^{t}\int_{0}^{T}\left|e^{f(u,\theta_{s}\omega)}-e^{f(u,\theta_{s}\omega_{t})}\right|\lambda(\theta_{s}\omega,u)duds\nonumber
\\
&\leq\frac{2 e^{\Vert f\Vert_{L^{\infty}}}}{T}\int_{t-T}^{t}\int_{0}^{T}\lambda(\theta_{s}\omega,u)duds\nonumber
\\
&=\frac{2 e^{\Vert f\Vert_{L^{\infty}}}}{T}\int_{t-T}^{t}\int_{0}^{T}\lambda\left(\sum_{\tau\in\omega[0,u+s)}h(u+s-\tau)\right)duds\nonumber
\\
&\leq 2 e^{\Vert f\Vert_{L^{\infty}}}TC_{\epsilon}
+\frac{2 e^{\Vert f\Vert_{L^{\infty}}}}{T}\epsilon\int_{t-T}^{t}\int_{0}^{T}\sum_{\tau\in\omega[0,u+s)}h(u+s-\tau)duds\nonumber
\\
&\leq 2 e^{\Vert f\Vert_{L^{\infty}}}TC_{\epsilon}+\frac{2 e^{\Vert f\Vert_{L^{\infty}}}}{T}\epsilon\int_{t-T}^{t}\int_{0}^{T}N[0,u+s]h(0)duds\nonumber
\\
&\leq 2 e^{\Vert f\Vert_{L^{\infty}}}TC_{\epsilon}+2 e^{\Vert f\Vert_{L^{\infty}}}\epsilon\int_{t-T}^{t}N[0,s+T]h(0)ds\nonumber
\\
&\leq 2 e^{\Vert f\Vert_{L^{\infty}}}TC_{\epsilon}+2 e^{\Vert f\Vert_{L^{\infty}}}\epsilon T(N[0,t]+N[t,t+T])h(0).\nonumber
\end{align}
Therefore,
\begin{equation}
\limsup_{t\rightarrow\infty}\frac{1}{t}\log\mathbb{E}^{P^{\emptyset}}
\left[e^{\frac{1}{T}\int_{0}^{t}\int_{0}^{T}\left|e^{f(u,\theta_{s}\omega)}
-e^{f(u,\theta_{s}\omega_{t})}\right|\lambda(\theta_{s}\omega,u)duds}\right]\leq c(\epsilon),
\end{equation}
where $c(\epsilon)\rightarrow 0$ as $\epsilon\rightarrow 0$. Since it is bounded above for any $c(\epsilon)$, it is bounded above by $0$.

For the second term,
\begin{align}
&\frac{1}{T}\int_{0}^{t}\int_{0}^{T}(e^{f(u,\theta_{s}\omega_{t})}-1)\left|\lambda(\theta_{s}\omega_{t},u)-\lambda(\theta_{s}\omega,u)\right|duds
\\
&\leq\frac{e^{\Vert f\Vert_{L^{\infty}}}+1}{T}\int_{0}^{t}\int_{0}^{T}\alpha\left|\sum_{\tau\in\omega_{t}[0,u+s)
\cup(\omega_{t})^{-}}h(u+s-\tau)-\sum_{\tau\in\omega[0,u+s)}h(u+s-\tau)\right|duds\nonumber
\\
&\leq\frac{e^{\Vert f\Vert_{L^{\infty}}}+1}{T}\int_{t-T}^{t}\int_{0}^{T}\alpha\sum_{\tau\in\omega_{t}[0,u+s)}h(u+s-\tau)duds\nonumber
\\
&+\frac{e^{\Vert f\Vert_{L^{\infty}}}+1}{T}\int_{t-T}^{t}\int_{0}^{T}\alpha\sum_{\tau\in\omega[0,u+s)}h(u+s-\tau)duds\nonumber
\\
&+\frac{e^{\Vert f\Vert_{L^{\infty}}}+1}{T}\int_{0}^{t}\int_{0}^{T}\alpha\sum_{\tau\in(\omega_{t})^{-}}h(u+s-\tau)duds\nonumber
\end{align}
Assume that $h(\cdot)$ is decreasing and $\lim_{z\rightarrow\infty}\frac{\lambda(z)}{z}=0$. 
By applying Jensen's inequality twice, we can estimate the second term above,
\begin{align}
&\mathbb{E}^{P^{\emptyset}}\left[e^{\frac{e^{\Vert f\Vert_{L^{\infty}}}+1}{T}\alpha\int_{t-T}^{t}\int_{0}^{T}\int_{0}^{u+s}h(u+s-v)dN_{v}duds}\right]
\\
&\leq\frac{1}{T}\int_{t-T}^{t}\mathbb{E}^{P^{\emptyset}}\left[e^{(e^{\Vert f\Vert_{L^{\infty}}}+1)
\alpha\int_{0}^{T}\int_{0}^{u+s}h(u+s-v)dN_{v}du}\right]ds\nonumber
\\
&\leq\frac{1}{T^{2}}\int_{t-T}^{t}\int_{0}^{T}\mathbb{E}^{P^{\emptyset}}\left[e^{(e^{\Vert f\Vert_{L^{\infty}}}+1)
\alpha T\int_{0}^{u+s}h(u+s-v)dN_{v}}\right]duds\nonumber
\\
&\leq\frac{1}{T^{2}}\int_{t-T}^{t}\int_{0}^{T}\mathbb{E}^{P^{\emptyset}}\left[e^{\int_{0}^{u+s}C(\alpha,T,h)\lambda(v)dv}\right]^{1/2}duds\nonumber
\\
&\leq\frac{e^{C(\alpha,T,h)C_{\epsilon}}}{T^{2}}\int_{t-T}^{t}
\int_{0}^{T}\mathbb{E}^{P^{\emptyset}}\left[e^{\epsilon C(\alpha,T,h)N[0,u+s]h(0)}\right]^{1/2}duds\nonumber
\\
&\leq e^{C(\alpha,T,h)C_{\epsilon}}\mathbb{E}^{P^{\emptyset}}\left[e^{\epsilon C(\alpha,T,h)N[0,t+T]h(0)}\right]^{1/2}.\nonumber
\end{align}
where $C(\alpha,T,h)=\exp(2(e^{\Vert f\Vert_{L^{\infty}}}+1)\alpha T h(0))-1$. Thus,
\begin{equation}
\limsup_{t\rightarrow\infty}\frac{1}{t}\log\mathbb{E}^{P^{\emptyset}}
\left[e^{\frac{e^{\Vert f\Vert_{L^{\infty}}}+1}{T}\alpha\int_{t-T}^{t}\int_{0}^{T}\int_{0}^{u+s}h(u+s-v)dN_{v}duds}\right]=0.
\end{equation}
Similarly, we can estimate the first term.

For the third term, by Jensen's inequality, we have
\begin{align}
&\mathbb{E}^{P^{\emptyset}}\left[e^{\frac{e^{\Vert f\Vert_{L^{\infty}}}+1}{T}
\int_{0}^{t}\int_{0}^{T}\alpha\sum_{\tau\in(\omega_{t})^{-}}h(u+s-\tau)duds}\right]
\\
&\leq\frac{1}{T}\int_{0}^{T}\mathbb{E}^{P^{\emptyset}}
\left[e^{\alpha(\exp(\Vert f\Vert_{L^{\infty}})+1)\int_{0}^{t}\sum_{\tau\in(\omega_{t})^{-}}h(u+s-\tau)ds}\right]du\nonumber
\\
&\leq\mathbb{E}^{P^{\emptyset}}\left[e^{\alpha(\exp(\Vert f\Vert_{L^{\infty}})+1)\int_{0}^{t}\sum_{\tau\in(\omega_{t})^{-}}h(s-\tau)ds}\right]\nonumber
\\
&=\mathbb{E}^{P^{\emptyset}}\left[e^{\alpha(\exp(\Vert f\Vert_{L^{\infty}})+1)\int_{0}^{t}\int_{0}^{t}\sum_{k=0}^{\infty}h(s+kt+t-u)dN_{u}ds}\right]\nonumber
\end{align}
Since we assume that $h(\cdot)$ is decreasing, $\int_{kt}^{(k+1)t}h(s)ds\geq th((k+1)t)$. Thus
\begin{equation}
\sum_{k=0}^{\infty}h(s+kt+t-u)\leq h(s+t-u)+\frac{1}{t}\int_{s+t-u}^{\infty}h(v)dv.
\end{equation}
Let $C(\alpha,f)=\alpha(\exp(\Vert f\Vert_{L^{\infty}})+1)$ and $H(t)=\int_{t}^{\infty}h(s)ds$. Then,
\begin{align}
&\mathbb{E}^{P^{\emptyset}}\left[e^{\alpha(\exp(\Vert f\Vert_{L^{\infty}})+1)\int_{0}^{t}\int_{0}^{t}\sum_{k=0}^{\infty}h(s+kt+t-u)dN_{u}ds}\right]
\\
&\leq\mathbb{E}^{P^{\emptyset}}
\left[e^{C(\alpha,f)\int_{0}^{t}\int_{0}^{t}\frac{1}{t}H(s+t-u)dN_{u}ds+C(\alpha,f)\int_{0}^{t}
\left[\int_{0}^{t}h(s+t-u)ds\right]dN_{u}}\right]\nonumber
\\
&=\mathbb{E}^{P^{\emptyset}}\left[e^{\int_{0}^{t}\left[\frac{C(\alpha,f)}{t}\int_{0}^{t}H(s+t-u)ds\right]dN_{u}
+\int_{0}^{t}\left[\int_{0}^{t}C(\alpha,f)h(s+t-u)ds\right]dN_{u}}\right].\nonumber
\end{align}
Notice that
\begin{equation}
\mathbb{E}^{P^{\emptyset}}\left[e^{\int_{0}^{t}\left[\frac{C(\alpha,f)}{t}\int_{0}^{t}H(s+t-u)ds\right]dN_{u}}\right]
\leq\mathbb{E}^{P^{\emptyset}}\left[e^{\left[\frac{C(\alpha,f)}{t}\int_{0}^{t}H(s)ds\right]N_{t}}\right],
\end{equation}
where $\frac{C(\alpha,f)}{t}\int_{0}^{t}H(s)ds\rightarrow 0$ as $t\rightarrow\infty$, which implies that
\begin{equation}
\limsup_{t\rightarrow\infty}\frac{1}{t}\log\mathbb{E}^{P^{\emptyset}}\left[e^{\int_{0}^{t}
\left[\frac{C(\alpha,f)}{t}\int_{0}^{t}H(s+u)ds\right]dN_{u}}\right]=0.
\end{equation}
Moreover,
\begin{align}
&\mathbb{E}^{P^{\emptyset}}\left[e^{\int_{0}^{t}\left[\int_{0}^{t}C(\alpha,f)h(s+t-u)ds\right]dN_{u}}\right]
\\
&\leq\mathbb{E}^{P^{\emptyset}}\left[e^{\int_{0}^{t}(e^{2\int_{0}^{t}C(\alpha,f)h(s+t-u)ds}-1)\lambda(u)du}\right]^{1/2}\nonumber
\\
&\leq e^{\frac{1}{2}C_{\epsilon}\int_{0}^{t}(e^{2\int_{0}^{t}C(\alpha,f)h(s+t-u)ds}-1)du}
\mathbb{E}^{P^{\emptyset}}\left[e^{\int_{0}^{t}(e^{\int_{0}^{t}2C(\alpha,f)h(s+t-u)ds}-1)\epsilon\sum_{\tau<u}h(u-\tau)du}\right]^{1/2}\nonumber
\\
&\leq
e^{\frac{1}{2}C_{\epsilon}\int_{0}^{t}(e^{2C(\alpha,f)H(t-u)}-1)du}\mathbb{E}^{P^{\emptyset}}
\left[e^{\int_{0}^{t}(e^{2C(\alpha,f)\Vert h\Vert_{L^{1}}}-1)\epsilon\sum_{\tau<u}h(u-\tau)du}\right]^{1/2}\nonumber
\\
&\leq
e^{\frac{1}{2}C_{\epsilon}\int_{0}^{t}(e^{2C(\alpha,f)H(u)}-1)du}\mathbb{E}^{P^{\emptyset}}
\left[e^{\epsilon(e^{2C(\alpha,f)\Vert h\Vert_{L^{1}}}-1)\Vert h\Vert_{L^{1}}N_{t}}\right]^{1/2}\nonumber
\end{align}
Notice that it holds for any $\epsilon>0$ and $\frac{1}{t}\int_{0}^{t}(e^{2C(\alpha,f)H(u)}-1)du\rightarrow 0$ as $t\rightarrow\infty$, which implies that
\begin{equation}
\limsup_{t\rightarrow\infty}\frac{1}{t}\log\mathbb{E}^{P^{\emptyset}}\left[e^{\int_{0}^{t}\left[\int_{0}^{t}C(\alpha,f)h(s+t-u)ds\right]dN_{u}}\right]=0.
\end{equation}

Putting all these thing together and applying H\"{o}lder's inequality several times, we get, for any $q>0$, $T>0$ and $F\in\mathcal{C}_{T}$,
\begin{equation}
\limsup_{t\rightarrow\infty}\frac{1}{t}\log\mathbb{E}^{P^{\emptyset}}
\left[\exp\left\{q\left|\frac{1}{T}\int_{0}^{t}F(\theta_{s}\omega_{t})ds-\frac{1}{T}\int_{0}^{t}F(\theta_{s}\omega)ds\right|\right\}\right]=0.
\end{equation}
\end{proof}

\begin{lemma}\label{finitetoinfinite}
We have
\begin{equation}
\lim_{T\rightarrow\infty}\frac{1}{T}\sup_{F\in\mathcal{C}_{T}}\int_{\Omega}F(\omega)Q(d\omega)\geq H(Q).
\end{equation}
\end{lemma}

\begin{proof}
Assume $H(Q)<\infty$. For any $\epsilon>0$, there exists some $f_{\epsilon}$ such that
\begin{equation}
\mathbb{E}^{Q}\left[\int_{0}^{1}f_{\epsilon}dN_{s}-\int_{0}^{1}(e^{f_{\epsilon}}-1)\lambda ds\right]\geq H(Q)-\epsilon.
\end{equation}
We can find a sequence $f_{T}\in\mathcal{B}\left(\mathcal{F}^{-(T-1)}_{s}\right)
\cap C(\Omega\times\mathbb{R})\rightarrow f_{\epsilon}$ as $T\rightarrow\infty$. By Fatou's lemma,
\begin{align}
&\liminf_{T\rightarrow\infty}\frac{1}{T}\sup_{F\in\mathcal{C}_{T}}\int_{\Omega}F(\omega)Q(d\omega)
\\
&\geq\liminf_{T\rightarrow\infty}\mathbb{E}^{Q}\left[\int_{0}^{1}f_{T}dN_{s}-\int_{0}^{1}(e^{f_{T}}-1)\lambda ds\right]\geq H(Q)-\epsilon.\nonumber
\end{align}
If $H(Q)=\infty$, then, for any $M>0$,
there exists some $f_{M}$ such that
\begin{equation}
\mathbb{E}^{Q}\left[\int_{0}^{1}f_{M}dN_{s}-\int_{0}^{1}(e^{f_{M}}-1)\lambda ds\right]\geq M.
\end{equation}
Repeat the same argument as in the case that $H(Q)<\infty$.
\end{proof}

\begin{lemma}\label{uppercompact}
For any compact set $A$, 
\begin{equation}
\limsup_{t\rightarrow\infty}\frac{1}{t}\log P(R_{t,\omega}\in A)\leq-\inf_{Q\in A}H(Q).
\end{equation}
\end{lemma}

\begin{proof}
Notice that
\begin{equation}
\mathbb{E}^{P^{\emptyset}}\left[e^{N[0,t]}\right]\leq\mathbb{E}^{P^{\emptyset}}\left[e^{(e^{2}-1)\int_{0}^{t}\lambda(s)ds}\right]^{1/2}
\leq\mathbb{E}^{P^{\emptyset}}\left[e^{(e^{2}-1)\epsilon h(0)N[0,t]+C_{\epsilon}(e^{2}-1)}\right]^{1/2}.
\end{equation}
By choosing $\epsilon>0$ small enough, we have $\mathbb{E}^{P^{\emptyset}}[e^{N[0,t]}]\leq e^{Ct}$ for some constant $C>0$. Therefore
\begin{equation}
\limsup_{\ell\rightarrow\infty}\limsup_{t\rightarrow\infty}\frac{1}{t}\log P^{\emptyset}\left(N[0,t]>\ell t\right)=-\infty,
\end{equation}
which implies (by comparing $\int_{\Omega}N[0,1]dR_{t,\omega}$ and $N[0,t]/t$ and the superexponential estimates in Lemma \ref{indfinite})
\begin{equation}
\limsup_{\ell\rightarrow\infty}\limsup_{t\rightarrow\infty}\frac{1}{t}\log P^{\emptyset}\left(\int_{\Omega}N[0,1]dR_{t,\omega}>\ell\right)=-\infty.
\end{equation}
Therefore, we need only to consider compact sets $A$ such that for any $Q\in A$, $\mathbb{E}^{Q}[N[0,1]]<\infty$.

Now for any $A$ compact consisting of $Q$ with $\mathbb{E}^{Q}[N[0,1]]<\infty$ and for any $F\in\mathcal{C}_{T}$ and 
for any $p,q>1$, $\frac{1}{p}+\frac{1}{q}=1$, by H\"{o}lder's inequality, Chebychev's inequality, and Lemma \ref{expectation},
\begin{align}
&P^{\emptyset}(R_{t,\omega}\in A)
\\
&\leq\mathbb{E}^{P^{\emptyset}}\left[e^{\frac{1}{pT}\int_{0}^{t}F(\theta_{s}\omega_{t})ds}\right]
\cdot\exp\left\{-\frac{t}{pT}\inf_{Q\in A}\int_{\Omega}F(\omega)Q(d\omega)\right\}\nonumber
\\
&\leq\mathbb{E}^{P^{\emptyset}}\left[e^{\frac{1}{T}\int_{0}^{t}F(\theta_{s}\omega)ds}\right]^{1/p}
\mathbb{E}^{P^{\emptyset}}\left[e^{\frac{q}{pT}\left|\int_{0}^{t}F(\theta_{s}\omega_{t})ds
-\int_{0}^{t}F(\theta_{s}\omega)ds\right|}\right]^{1/q}\nonumber
\\
&\phantom{\mathbb{E}^{P^{\emptyset}}\left[e^{\frac{1}{T}\int_{0}^{t}F(\theta_{s}\omega)ds}\right]^{1/p}\mathbb{E}^{P^{\emptyset}}}
\cdot\exp\left\{-\frac{t}{pT}\inf_{Q\in A}\int_{\Omega}F(\omega)Q(d\omega)\right\}\nonumber
\\
&\leq\mathbb{E}^{P^{\emptyset}}\left[e^{\frac{q}{pT}\left|\int_{0}^{t}F(\theta_{s}\omega_{t})ds
-\int_{0}^{t}F(\theta_{s}\omega)ds\right|}\right]^{1/q}\cdot\exp\left\{-\frac{t}{pT}\inf_{Q\in A}\int_{\Omega}F(\omega)Q(d\omega)\right\}\nonumber
\end{align}
By Lemma \ref{difference}, we have
\begin{equation}
\limsup_{t\rightarrow\infty}\frac{1}{t}\log P^{\emptyset}(R_{t,\omega}\in A)\leq-\frac{1}{p}\inf_{Q\in A}\frac{1}{T}\int_{\Omega}F(\omega)Q(d\omega).
\end{equation}
Since it holds for any $p>1$, we get
\begin{equation}
\limsup_{t\rightarrow\infty}\frac{1}{t}\log P^{\emptyset}(R_{t,\omega}\in A)\leq-\inf_{Q\in A}\frac{1}{T}\int_{\Omega}F(\omega)Q(d\omega).
\end{equation}

For any compact $A$, given $Q\in A$ and $\epsilon>0$, by Lemma \ref{finitetoinfinite}, there exists $T_{Q}>0$ and 
$F_{Q}\in\mathcal{C}_{T_{Q}}$ such that $\frac{1}{T_{Q}}\int_{\Omega}F_{Q}(\omega)Q(d\omega)\geq\inf_{A\in Q}H(Q)-\frac{1}{2}\epsilon$. 
Since the linear integral is a continuous functional of $Q$ (see the proof of Lemma \ref{lscconvex}), 
there exists a neighborhood $G_{Q}$ of $Q$ such that $\frac{1}{T_{Q}}\int_{\Omega}F_{Q}(\omega)Q(d\omega)\geq\inf_{A\in Q}H(Q)-\epsilon$ 
for all $Q\in G_{Q}$. Since $A$ is compact, there exists $G_{Q_{1}},\ldots,G_{Q_{\ell}}$ such that $A\subset\bigcup_{j=1}^{\ell}G_{Q_{j}}$. 
Hence
\begin{equation}
\inf_{1\leq j\leq\ell}\sup_{T>0}\sup_{F\in\mathcal{C}_{T}}\inf_{Q\in G_{j}}\frac{1}{T}\int_{\Omega}F(\omega)Q(d\omega)\geq\inf_{Q\in A}H(Q)-\epsilon.
\end{equation}
Note that for any $A,B$,
\begin{align}
&\limsup_{t\rightarrow\infty}\frac{1}{t}\log P(R_{t,\omega}\in A\cup B)
\\
&\leq\max\left\{\limsup_{t\rightarrow\infty}\frac{1}{t}\log 
P(R_{t,\omega}\in A),\limsup_{t\rightarrow\infty}\frac{1}{t}\log P(R_{t,\omega}\in B)\right\}.\nonumber
\end{align}
Thus, for $A\subset\bigcup_{j=1}^{\ell}G_{j}$,
\begin{equation}
\limsup_{t\rightarrow\infty}\frac{1}{t}\log P(R_{t,\omega}\in A)\leq-\inf_{1\leq j\leq\ell}
\sup_{T>0}\sup_{F\in\mathcal{C}_{T}}\inf_{Q\in G_{j}}\frac{1}{T}\int F(\omega)Q(d\omega).
\end{equation}
Hence $\limsup_{t\rightarrow\infty}\frac{1}{t}\log P(R_{t,\omega}\in A)\leq-\inf_{Q\in A}H(Q)$ for any compact $A$.
\end{proof}

\begin{theorem}[Upper Bound]
For any closed set $C$, 
\begin{equation}
\limsup_{t\rightarrow\infty}\frac{1}{t}\log P(R_{t,\omega}\in C)\leq-\inf_{Q\in C}H(Q).
\end{equation}
\end{theorem}

\begin{proof}
For any closed set $C$ and compact $\mathcal{A}^{n}$ which is defined in Lemma \ref{tightness}, we have
\begin{align}
&\limsup_{t\rightarrow\infty}\frac{1}{t}\log P(R_{t,\omega}\in C)
\\
&\leq\max\left\{\limsup_{t\rightarrow\infty}\frac{1}{t}\log P(R_{t,\omega}\in C\cap\mathcal{A}^{n}),
\limsup_{t\rightarrow\infty}\frac{1}{t}\log P(R_{t,\omega}\in(\mathcal{A}^{n})^{c})\right\}.\nonumber
\end{align}
Since $C\cap\mathcal{A}_{n}$ is compact, by Lemma \ref{uppercompact},
\begin{equation}
\limsup_{t\rightarrow\infty}\frac{1}{t}\log P(R_{t,\omega}\in C\cap\mathcal{A}^{n})\leq-\inf_{Q\in C\cap\mathcal{A}^{n}}H(Q)\leq-\inf_{Q\in C}H(Q).
\end{equation}
Furthermore, by Lemma \ref{finalestimate}, we have
\begin{align}
&\limsup_{t\rightarrow\infty}\frac{1}{t}\log P(R_{t,\omega}\in(\mathcal{A}^{n})^{c})
\\
&=\limsup_{t\rightarrow\infty}\frac{1}{t}\log P\left(R_{t,\omega}\in\bigcup_{j=n}^{\infty}\mathcal{A}_{\frac{1}{j},j,j}^{c}\right)\nonumber
\\
&\leq\max_{j\geq n}\max\bigg\{\limsup_{t\rightarrow\infty}\frac{1}{t}
\log P\left(\frac{1}{t}\int_{0}^{t}\chi_{N[0,1]\geq j}(\theta_{s}\omega_{t})ds\geq\varepsilon(j)\right),\nonumber
\\
&\limsup_{t\rightarrow\infty}\frac{1}{t}\log P\left(\frac{1}{t}\int_{0}^{t}\chi_{N[0,1/j]\geq 2}(\theta_{s}\omega_{t})ds\geq (1/j)g(1/j)\right),\nonumber
\\
&\limsup_{t\rightarrow\infty}\frac{1}{t}\log 
P\left(\frac{1}{t}\int_{0}^{t}N[0,1]\chi_{N[0,1]\geq\ell}(\theta_{s}\omega_{t})ds\geq m(\ell)\right)\bigg\}\rightarrow-\infty\nonumber
\end{align}
as $n\rightarrow\infty$. Hence, we have
\begin{equation}
\limsup_{t\rightarrow\infty}\frac{1}{t}\log P(R_{t,\omega}\in C)\leq-\inf_{Q\in C}H(Q).
\end{equation}
\end{proof}

\section{Superexponential Estimates}\label{superexpestimates}

In order to get the full large deviation principle, we need the upper bound inequality valid for any closed set instead of for any compact set, 
which requires some superexponential estimates.

\begin{lemma}\label{reverse}
For any $q>0$,
\begin{equation}
\limsup_{t\rightarrow\infty}\frac{1}{t}\log\mathbb{E}^{P^{\emptyset}}\left[e^{q\int_{0}^{t}h(t-s)dN_{s}}\right]=0.
\end{equation}
\end{lemma}

\begin{proof}
\begin{align}
\mathbb{E}^{P^{\emptyset}}\left[e^{q\int_{0}^{t}h(t-s)dN_{s}}\right]
&\leq\mathbb{E}^{P^{\emptyset}}\left[e^{\int_{0}^{t}(e^{2qh(t-s)}-1)\lambda(\sum_{0<\tau<s}h(s-\tau))ds}\right]^{1/2}
\\
&\leq\mathbb{E}^{P^{\emptyset}}\left[e^{(C_{\epsilon}+h(0)\epsilon N_{t})\int_{0}^{t}(e^{2qh(t-s)}-1)ds}\right]^{1/2}.\nonumber
\end{align}
Note that $\int_{0}^{t}(e^{2qh(t-s)}-1)ds=\int_{0}^{t}(e^{2qh(s)}-1)ds\in L_{1}$ since $h\in L^{1}$. Therefore,
\begin{equation}
\limsup_{t\rightarrow\infty}\frac{1}{t}\log\mathbb{E}^{P^{\emptyset}}\left[e^{q\int_{0}^{t}h(t-s)dN_{s}}\right]\leq c(\epsilon),
\end{equation}
where $c(\epsilon)\rightarrow 0$ as $\epsilon\rightarrow 0$. Since it holds for any $\epsilon$, we get the desired result.
\end{proof}

\begin{lemma}\label{indfinite}
For any $q>0$ and $T>0$,
\begin{equation}
\limsup_{t\rightarrow\infty}\frac{1}{t}\log\mathbb{E}^{P^{\emptyset}}\left[e^{qN[t,t+T]}\right]=0.
\end{equation}
Therefore, for any $\epsilon>0$,
\begin{equation}
\limsup_{t\rightarrow\infty}\frac{1}{t}\log P^{\emptyset}\left(N[t,t+T]\geq\epsilon t\right)=-\infty.
\end{equation}
\end{lemma}

\begin{proof}
By H\"{o}lder's inequality,
\begin{align}
\mathbb{E}^{P^{\emptyset}}\left[e^{qN[t,t+T]}\right]&\leq\mathbb{E}^{P^{\emptyset}}
\left[e^{(e^{2q}-1)\int_{t}^{t+T}\lambda(\sum_{0<\tau<s}h(s-\tau))ds}\right]^{1/2}
\\
&\leq e^{\frac{1}{2}(e^{2q}-1)C_{\epsilon}T}\cdot\mathbb{E}^{P^{\emptyset}}
\left[e^{\epsilon(e^{2q}-1)h(0)N[t,t+T]+\epsilon(e^{2q}-1)\int_{0}^{t}h(t-s)dN_{s}}\right]^{1/2}\nonumber
\\
&\leq e^{\frac{1}{2}(e^{2q}-1)C_{\epsilon}T}\cdot\mathbb{E}^{P^{\emptyset}}
\left[e^{2\epsilon(e^{2q}-1)h(0)N[t,t+T]}\right]^{1/4}\mathbb{E}^{P^{\emptyset}}\nonumber
\\
&\phantom{\leq e^{\frac{1}{2}(e^{2q}-1)C_{\epsilon}T}\cdot\mathbb{E}^{P^{\emptyset}}}
\cdot\left[e^{2\epsilon(e^{2q}-1)\int_{0}^{t}h(t-s)dN_{s}}\right]^{1/4}.\nonumber
\end{align}
Choose $\epsilon<q[2(e^{2q}-1)h(0)]^{-1}$. Then
\begin{equation}
\mathbb{E}^{P^{\emptyset}}\left[e^{qN[t,t+T]}\right]^{3/4}\leq e^{\frac{1}{2}(e^{2q}-1)C_{\epsilon}T}
\cdot\mathbb{E}^{P^{\emptyset}}\left[e^{2\epsilon(e^{2q}-1)\int_{0}^{t}h(t-s)dN_{s}}\right]^{1/4}.
\end{equation}
By Lemma \ref{reverse}, we get the desired result.
\end{proof}

\begin{lemma}\label{midestimate}
We have the following superexponential estimates.

(i) For any $\epsilon>0$,
\begin{equation}
\limsup_{\delta\rightarrow 0}\limsup_{t\rightarrow\infty}\frac{1}{t}\log P\left(\frac{1}{\delta t}
\int_{0}^{t}\chi_{N[0,\delta]\geq 2}(\theta_{s}\omega)ds\geq\epsilon\right)=-\infty.
\end{equation}

(ii) For any $\epsilon>0$,
\begin{equation}
\limsup_{M\rightarrow\infty}\limsup_{t\rightarrow\infty}\frac{1}{t}
\log P\left(\frac{1}{t}\int_{0}^{t}\chi_{N[0,1]\geq M}(\theta_{s}\omega)ds\geq\epsilon\right)=-\infty.
\end{equation}

(iii) For any $\epsilon>0$,
\begin{equation}
\limsup_{\ell\rightarrow\infty}\limsup_{t\rightarrow\infty}\frac{1}{t}
\log P\left(\frac{1}{t}\int_{0}^{t}N[0,1]\chi_{N[0,1]\geq\ell}(\theta_{s}\omega)ds\geq \epsilon\right)=-\infty.
\end{equation}
\end{lemma}

\begin{proof}
(i) Let us define
\begin{equation}
N_{\ell'}[0,t]=\int_{0}^{t}\chi_{\lambda(s)<\ell'}dN_{s},\quad\hat{N}_{\ell'}[0,t]=\int_{0}^{t}\chi_{\lambda(s)\geq\ell'}dN_{s}.
\end{equation}
Then $N[0,t]=N_{\ell'}[0,t]+\hat{N}_{\ell'}[0,t]$ and $N_{\ell'}[0,t]$ has compensator 
$\int_{0}^{t}\lambda(s)\chi_{\lambda(s)<\ell'}ds$ and $\hat{N}_{\ell'}[0,t]$ has 
compensator $\int_{0}^{t}\lambda(s)\chi_{\lambda(s)\geq\ell'}ds$. Notice that
\begin{equation}
\chi_{N[0,\delta]\geq 2}\leq\chi_{N_{\ell'}[0,\delta]\geq 2}+\chi_{\hat{N}_{\ell'}[0,\delta]\geq 1}.
\end{equation}
It is clear that $N_{\ell'}$ is dominated by the usual Poisson process with rate $\ell'$. By Lemma \ref{constrate},
\begin{equation}
\limsup_{\delta\rightarrow 0}\limsup_{t\rightarrow\infty}\frac{1}{t}
\log P\left(\frac{1}{\delta t}\int_{0}^{t}\chi_{N_{\ell'}[0,\delta]\geq 2}(\theta_{s}\omega)ds\geq\frac{\epsilon}{2}\right)=-\infty.
\end{equation}
On the other hand,
\begin{align}
\frac{1}{\delta}\int_{0}^{t}\chi_{\hat{N}_{\ell'}[0,\delta]\geq 1}(\theta_{s}\omega)ds
&=\frac{1}{\delta}\int_{0}^{t}\chi_{\hat{N}_{\ell'}[s,s+\delta]\geq 1}(\omega)ds
\\
&\leq\frac{1}{\delta}\int_{0}^{t}\hat{N}_{\ell'}[s,s+\delta]ds\nonumber
\\
&=\frac{1}{\delta}\int_{\delta}^{t+\delta}\hat{N}_{\ell'}[0,s]ds-\frac{1}{\delta}\int_{0}^{t}\hat{N}_{\ell'}[0,s]ds\nonumber
\\
&\leq\hat{N}_{\ell'}[0,t]+N[t,t+\delta].\nonumber
\end{align}
By Lemma \ref{indfinite}, we have
\begin{equation}
\limsup_{t\rightarrow\infty}\frac{1}{t}\log P\left(\frac{1}{t}N[t,t+\delta]\geq\frac{\epsilon}{4}\right)=-\infty,
\end{equation}
for any $\delta>0$. Hence
\begin{equation}
\limsup_{\delta\rightarrow 0}\limsup_{t\rightarrow\infty}\frac{1}{t}\log P\left(\frac{1}{t}N[t,t+\delta]\geq\frac{\epsilon}{4}\right)=-\infty.
\end{equation}
Finally, for some positive $h(\ell')$ to be chosen later, 
\begin{align}
P\left(\frac{1}{t}\hat{N}_{\ell'}[0,t]\geq\frac{\epsilon}{4}\right)&\leq\mathbb{E}\left[e^{h(\ell')\hat{N}_{\ell'}[0,t]}\right]e^{-t h(\ell')\epsilon/4}
\\
&\leq\mathbb{E}\left[e^{(e^{2h(\ell')}-1)\int_{0}^{t}\lambda(s)\chi_{\lambda(s)\geq\ell'}ds}\right]^{1/2}e^{-t h(\ell')\epsilon/4}.\nonumber
\end{align}
Let $f(z)=\frac{z}{\lambda(z)}$. Then $f(z)\rightarrow\infty$ as $z\rightarrow\infty$. Let $Z_{s}=\sum_{\tau\in\omega[0,s]}h(s-\tau)$. 
Then, by the definition of $\lambda(s)$ and abusing the notation a little bit, we see that $\lambda(s)=\lambda(Z_{s})$. 
Since $\lambda(\cdot)$ is increasing, its inverse function $\lambda^{-1}$ exists and $\lambda^{-1}(\ell')\rightarrow\infty$ 
as $\ell'\rightarrow\infty$. We have
\begin{align}
\mathbb{E}\left[e^{(e^{2h(\ell')}-1)\int_{0}^{t}\lambda(s)\chi_{\lambda(s)\geq\ell'}ds}\right]^{1/2}
&\leq\mathbb{E}\left[e^{(e^{2h(\ell')}-1)\int_{0}^{t}\lambda(Z_{s})\chi_{Z_{s}\geq\lambda^{-1}(\ell')}ds}\right]^{1/2}
\\
&\leq\mathbb{E}\left[e^{(e^{2h(\ell')}-1)\int_{0}^{t}\lambda(Z_{s})\frac{f(Z_{s})}{\inf_{z\geq\ell'}f(\lambda^{-1}(z))}ds}\right]^{1/2}.
\end{align}
It is clear that $\lim_{\ell'\rightarrow\infty}\inf_{z\geq\ell'}f(\lambda^{-1}(z))=\infty$. Choose
\begin{equation}
h(\ell')=\frac{1}{2}\log\left[\inf_{z\geq\ell'}f(\lambda^{-1}(z))+1\right].
\end{equation}
Then $h(\ell')\rightarrow\infty$ as $\ell'\rightarrow\infty$ and
\begin{align}
\mathbb{E}\left[e^{(e^{2h(\ell')}-1)\int_{0}^{t}\lambda(Z_{s})\frac{f(Z_{s})}{\inf_{z\geq\ell'}f(\lambda^{-1}(z))}ds}\right]^{1/2}
&=\mathbb{E}\left[e^{\int_{0}^{t}Z_{s}ds}\right]^{1/2}
\\
&=\mathbb{E}\left[e^{\int_{0}^{t}\sum_{\tau\in\omega[0,s]}h(s-\tau)ds}\right]^{1/2}\nonumber
\\
&\leq\mathbb{E}\left[e^{\Vert h\Vert_{L^{1}}N_{t}}\right]^{1/2}.\nonumber
\end{align}
Hence,
\begin{equation}
\limsup_{\ell'\rightarrow\infty}\limsup_{\delta\rightarrow 0}\limsup_{t\rightarrow\infty}\frac{1}{t}
\log P\left(\frac{1}{t}\hat{N}_{\ell'}[0,t]\geq\frac{\epsilon}{4}\right)=-\infty.
\end{equation}

(ii) It is easy to see that (iii) implies (ii).

(iii) Observe first that
\begin{align}
N[s,s+1]\chi_{N[s,s+1]\geq\ell}&\leq N_{\ell'}[s,s+1]\chi_{N_{\ell'}[s,s+1]\geq\frac{\ell}{2}}
+\hat{N}_{\ell'}[s,s+1]\chi_{\hat{N}_{\ell'}[s,s+1]\geq\frac{\ell}{2}}
\\
&+\frac{\ell}{2}\chi_{N_{\ell'}[s,s+1]\geq\frac{\ell}{2}}+\frac{\ell}{2}\chi_{\hat{N}_{\ell'}[s,s+1]\geq\frac{\ell}{2}}.\nonumber
\end{align}
For the first term, notice that $N_{\ell'}$ is dominated by a usual Poisson process with rate $\ell'$. 
Thus, by Lemma \ref{constrateIII},
\begin{equation}
\limsup_{\ell\rightarrow\infty}\limsup_{t\rightarrow\infty}\frac{1}{t}
\log P\left(\frac{1}{t}\int_{0}^{t}N_{\ell'}[s,s+1]\chi_{N_{\ell'}[s,s+1]\geq\frac{\ell}{2}}(\omega)ds\geq\frac{\epsilon}{4}\right)=-\infty.
\end{equation}
For the second term, $\hat{N}_{\ell'}[s,s+1]\chi_{\hat{N}_{\ell'}[s,s+1]\geq\frac{\ell}{2}}\leq\hat{N}_{\ell'}[s,s+1]$ and
\begin{equation}
\int_{0}^{t}\hat{N}_{\ell'}[s,s+1]ds\leq\hat{N}_{\ell'}[0,t]+N[t,t+1].
\end{equation}
By Lemma \ref{indfinite},
\begin{equation}
\limsup_{t\rightarrow\infty}\frac{1}{t}\log P\left(\frac{1}{t}N[t,t+1]\geq\frac{\epsilon}{8}\right)=-\infty,
\end{equation}
and by the same argument as in (i),
\begin{equation}
\limsup_{\ell'\rightarrow\infty}\limsup_{\ell\rightarrow\infty}\limsup_{t\rightarrow\infty}\frac{1}{t}
\log P\left(\frac{1}{t}\hat{N}_{\ell'}[0,t]\geq\frac{\epsilon}{8}\right)=-\infty.
\end{equation}
For the third term, notice that
\begin{equation}
\int_{0}^{t}\frac{\ell}{2}\chi_{N_{\ell'}[s,s+1]\geq\frac{\ell}{2}}ds\leq\int_{0}^{t}N_{\ell'}[s,s+1]
\chi_{N_{\ell'}[s,s+1]\geq\frac{\ell}{2}}(\omega)ds.
\end{equation}
We can get the same superexponential estimate as before. Finally, for the fourth term,
\begin{equation}
\int_{0}^{t}\frac{\ell}{2}\chi_{\hat{N}_{\ell'}[s,s+1]\geq\frac{\ell}{2}}ds\leq\int_{0}^{t}\hat{N}_{\ell'}[s,s+1](\omega)ds.
\end{equation}
We can get the same superexponential estimate as before.
\end{proof}

\begin{lemma}\label{constrate}
Assume $N_{t}$ is a usual Poisson process with constant rate $\lambda$. Then, for any $\epsilon>0$
\begin{equation}
\limsup_{\delta\rightarrow 0}\limsup_{t\rightarrow\infty}\frac{1}{t}\log\mathbb{P}\left(\frac{1}{\delta t}
\int_{0}^{t}\chi_{N[s,s+\delta]\geq 2}(\omega)ds\geq\epsilon\right)=-\infty.
\end{equation}
\end{lemma}

\begin{proof}
Let $f(\delta,\omega)=\frac{1}{h(\delta)}\chi_{N[0,\delta]\geq 2}(\omega)$, where $h(\delta)$ is to be chosen later. 
By Jensen's inequality and stationarity and independence of increments of Poisson process, we have
\begin{align}
\mathbb{E}\left[e^{\int_{0}^{t}\frac{1}{\delta}f(\delta,\theta_{s}\omega)ds}\right]
&\leq\mathbb{E}\left[e^{\frac{1}{\delta}\int_{0}^{\delta}\sum_{j=0}^{[t/\delta]}f(\delta,\theta_{s+j\delta}\omega)ds}\right]
\\
&\leq\mathbb{E}\left[\frac{1}{\delta}\int_{0}^{\delta}e^{\sum_{j=0}^{[t/\delta]}f(\delta,\theta_{s+j\delta}\omega)}ds\right]\nonumber
\\
&=\mathbb{E}\left[e^{\sum_{j=0}^{[t/\delta]}f(\delta,\theta_{j\delta}\omega)}\right]\nonumber
\\
&=\mathbb{E}\left[e^{f(\delta,\omega)}\right]^{[t/\delta]+1}\nonumber
\\
&=\left\{e^{1/h(\delta)}(1-e^{-\lambda\delta}-\lambda\delta e^{-\lambda\delta})+e^{-\lambda\delta}
+\lambda\delta e^{-\lambda\delta}\right\}^{[t/\delta]+1}\nonumber
\\
&\leq (M'e^{1/h(\delta)}\lambda^{2}\delta^{2}+1)^{[t/\delta]+1},\nonumber
\end{align}
for some $M'>0$. Choose $h(\delta)=\frac{1}{\log(1/\delta)}$. Thus,
\begin{equation}
\mathbb{E}\left[e^{\int_{0}^{t}\frac{1}{\delta}f(\delta,\theta_{s}\omega)ds}\right]\leq(M'\delta+1)^{[t/\delta]+1}\leq e^{Mt},
\end{equation}
for some $M>0$. Therefore, by Chebychev's inequality,
\begin{equation}
\limsup_{t\rightarrow\infty}\frac{1}{t}\log\mathbb{P}\left(\frac{1}{\delta h(\delta)t}\int_{0}^{t}
\chi_{N[s,s+\delta]\geq 2}(\omega)ds\geq\frac{\epsilon}{h(\delta)}\right)
\leq M-\frac{\epsilon}{h(\delta)},
\end{equation}
which holds for any $\delta>0$. Letting $\delta\rightarrow 0$, we get the desired result.
\end{proof}

\begin{lemma}\label{constrateIII}
Assume $N_{t}$ is a usual Poisson process with constant rate $\lambda$. Then, for any $\epsilon>0$,
\begin{equation}
\limsup_{\ell\rightarrow\infty}\limsup_{t\rightarrow\infty}\frac{1}{t}
\log P\left(\frac{1}{t}\int_{0}^{t}N[0,1]\chi_{N[0,1]\geq\ell}(\theta_{s}\omega)ds\geq \epsilon\right)=-\infty.
\end{equation}
\end{lemma}

\begin{proof}
Let $h(\ell)$ be some function of $\ell$ to be chosen later. Following the same argument as in the proof of Lemma \ref{constrate}, we have
\begin{align}
&\mathbb{P}\left(h(\ell)\int_{0}^{t}N[0,1]\chi_{N[0,1]\geq\ell}(\theta_{s}\omega)ds\geq\epsilon h(\ell)t\right)
\\
&\leq\mathbb{E}\left[e^{h(\ell)\int_{0}^{t}N[0,1]\chi_{N[0,1]\geq\ell}(\theta_{s}\omega)ds}\right]e^{-\epsilon h(\ell)t}\nonumber
\\
&\leq\mathbb{E}\left[e^{h(\ell)N[0,1]\chi_{N[0,1]\geq\ell}}\right]^{[t]+1}e^{-\epsilon h(\ell)t}\nonumber
\\
&=\left\{\mathbb{P}(N[0,1]<\ell)+\sum_{k=\ell}^{\infty}e^{h(\ell)k}e^{-\lambda}\frac{\lambda^{k}}{k!}\right\}^{[t]+1}e^{-\epsilon h(\ell)t}\nonumber
\\
&\leq\left\{1+C_{1}\sum_{k=\ell}^{\infty}e^{h(\ell)k+\log(\lambda)k-\log(k)k}\right\}^{[t]+1}e^{-\epsilon h(\ell)t}\nonumber
\\
&\leq\left\{1+C_{2}e^{h(\ell)\ell+\log(\lambda)\ell-\log(\ell)\ell}\right\}^{[t]+1}e^{-\epsilon h(\ell)t}.\nonumber
\end{align}
Choosing $h(\ell)=(\log(\ell))^{1/2}$ will do the work.
\end{proof}

The following Lemma \ref{finalestimate} provides us the superexponential estimates that we need. These superexponential
estimates have basically been done in Lemma \ref{midestimate}. The difference is that in the statement in Lemma \ref{midestimate},
we used $\omega$ and in Lemma \ref{finalestimate} it is changed to $\omega_{t}$ which is what we needed.
Lemma \ref{finalestimate} has three statements. Part (i) says if you start with a sequence of simple point processes,
the limiting point process may not be simple, but that has probability that is superexponentially small. Part (ii) is the
usual superexponential we would expect if $\mathcal{M}_{S}(\Omega)$ were equipped with weak topology. But since
we are using a strengthened weak topology with the convergence of first moment as well, we will also need Part (iii).

\begin{lemma}\label{finalestimate}
We have the following superexponential estimates.

(i) For some $g(\delta)\rightarrow 0$ as $\delta\rightarrow 0$,
\begin{equation}
\limsup_{\delta\rightarrow 0}\limsup_{t\rightarrow\infty}\frac{1}{t}\log P\left(\frac{1}{\delta t}
\int_{0}^{t}\chi_{N[0,\delta]\geq 2}(\theta_{s}\omega_{t})ds\geq g(\delta)\right)=-\infty.
\end{equation}

(ii) For some $\varepsilon(M)\rightarrow 0$ as $M\rightarrow\infty$,
\begin{equation}
\limsup_{M\rightarrow\infty}\limsup_{t\rightarrow\infty}\frac{1}{t}\log P\left(\frac{1}{t}\int_{0}^{t}
\chi_{N[0,1]\geq M}(\theta_{s}\omega_{t})ds\geq\varepsilon(M)\right)=-\infty.
\end{equation}

(iii) For some $m(\ell)\rightarrow 0$ as $\ell\rightarrow\infty$,
\begin{equation}
\limsup_{\ell\rightarrow\infty}\limsup_{t\rightarrow\infty}\frac{1}{t}\log P\left(\frac{1}{t}
\int_{0}^{t}N[0,1]\chi_{N[0,1]\geq\ell}(\theta_{s}\omega_{t})ds\geq m(\ell)\right)=-\infty.
\end{equation}
\end{lemma}

\begin{proof}
We can replace the $\epsilon$ in the statement of Lemma \ref{midestimate} by $g(\delta)$, $\varepsilon(M)$ 
and $m(\ell)$ by a standard analysis argument. We can also replace the $\omega$ in Lemma \ref{midestimate} by $\omega_{t}$ here since
\begin{equation}
\left|\int_{0}^{t}\chi_{N[0,\delta]\geq 2}(\theta_{s}\omega_{t})ds-\int_{0}^{t}\chi_{N[0,\delta]\geq 2}(\theta_{s}\omega)ds\right|\leq 2\delta,
\end{equation}
\begin{equation}
\left|\int_{0}^{t}\chi_{N[0,1]\geq M}(\theta_{s}\omega_{t})ds-\int_{0}^{t}\chi_{N[0,1]\geq M}(\theta_{s}\omega)ds\right|\leq 2,
\end{equation}
and
\begin{align}
&\left|\int_{0}^{t}N[0,1]\chi_{N[0,1]\geq\ell}(\theta_{s}\omega_{t})ds-\int_{0}^{t}N[0,1]\chi_{N[0,1]\geq\ell}(\theta_{s}\omega)ds\right|
\\
&\leq\int_{t-1}^{t}N[s,s+1](\omega)ds+\int_{t-1}^{t}N[s,s+1](\omega_{t})ds\nonumber
\\
&\leq N[t-1,t+1](\omega)+N[t-1,t+1](\omega_{t})\nonumber
\\
&=N[t-1,t+1](\omega)+N[t-1,t](\omega)+N[0,1](\omega).\nonumber
\end{align}
By Lemma \ref{indfinite}, we have the superexponential estimate, for any $\epsilon>0$,
\begin{equation}
\limsup_{t\rightarrow\infty}\frac{1}{t}\log P\left(\frac{1}{t}\left\{N[t-1,t+1](\omega)+N[t-1,t](\omega)+N[0,1](\omega)\right\}\geq\epsilon\right)=-\infty.
\end{equation}
\end{proof}

\begin{lemma}\label{tightness}
For any $\delta, M>0,\ell>0$, define
\begin{align}
&\mathcal{A}_{\delta}=\left\{Q\in\mathcal{M}_{S}(\Omega): Q(N[0,\delta]\geq 2)\leq\delta g(\delta)\right\},
\\
&\mathcal{A}_{M}=\left\{Q\in\mathcal{M}_{S}(\Omega): Q(N[0,1]\geq M)\leq\varepsilon(M)\right\},\nonumber
\\
&\mathcal{A}_{\ell}=\left\{Q\in\mathcal{M}_{S}(\Omega):\int_{N[0,1]\geq\ell}N[0,1]dQ\leq m(\ell)\right\},\nonumber
\end{align}
where $\varepsilon(M)\rightarrow 0$ as $M\rightarrow\infty$, $m(\ell)\rightarrow 0$ as $\ell\rightarrow\infty$ 
and $g(\delta)\rightarrow 0$ as $\delta\rightarrow 0$. 
Let $\mathcal{A}_{\delta,M,\ell}=\mathcal{A}_{\delta}\cap\mathcal{A}_{M}\cap\mathcal{A}_{\ell}$ and
\begin{equation}
\mathcal{A}^{n}=\bigcap_{j=n}^{\infty}\mathcal{A}_{\frac{1}{j},j,j}.
\end{equation}
Then, $\mathcal{A}^{n}$ is compact.
\end{lemma}

\begin{proof}
Observe that for $\beta>0$,
\begin{equation}
K_{\beta}=\bigcap_{k=1}^{\infty}\left\{\omega:\{N[-k,-(k-1)](\omega)\leq\beta\ell_{k}\}\cap\{N[k-1,k](\omega)\leq\beta\ell_{k}\}\right\}
\end{equation}
are relatively compact sets in $\Omega$. Let $\overline{K_{\beta}}$ be the closure of $K_{\beta}$, 
which is then compact.

For any $Q\in\mathcal{A}^{n}$, $Q(N[0,1]\geq M)\leq\epsilon(M)$ for any $M\geq n$. 
We can choose $\beta$ big enough and an increasing sequence $\ell_{k}$ such that $\beta\ell_{1}\geq n$ 
and $\infty>\sum_{k=1}^{\infty}\epsilon(\beta\ell_{k})\rightarrow 0$ as $\beta\rightarrow\infty$, uniformly for $Q\in\mathcal{A}^{n}$,
\begin{align}
Q\left(\overline{K_{\beta}}^{c}\right)&\leq Q(K_{\beta}^{c})
\\
&=Q\left(\bigcup_{k=1}^{\infty}\{N[-k,-(k-1)](\omega)>\beta\ell_{k}\}\cap\{N[k-1,k](\omega)>\beta\ell_{k}\}\right)\nonumber
\\
&\leq\sum_{k=1}^{\infty}\left\{Q(N[-(k-1),-k]>\beta\ell)+Q(N[k-1,k]>\beta\ell_{k})\right\}\nonumber
\\
&=2\sum_{k=1}^{\infty}Q(N[0,1]>\beta\ell_{k})\nonumber
\\
&\leq 2\sum_{k=1}^{\infty}\epsilon(\beta\ell_{k})\rightarrow 0\nonumber
\end{align}
as $\beta\rightarrow\infty$. Therefore, $\mathcal{A}^{n}$ is tight in the weak topology and by 
Prokhorov theorem $\mathcal{A}^{n}$ is precompact in the weak topology. 
In other words, for any sequence in $\mathcal{A}^{n}$, there exists a subsequence, say $Q_{n}$ 
such that $Q_{n}\rightarrow Q$ weakly as $n\rightarrow\infty$ for some $Q$. By the definition of $\mathcal{A}^{n}$, $Q_{n}$ 
are uniformly integrable, which implies that $\int N[0,1]dQ_{n}\rightarrow\int N[0,1]dQ$ 
as $n\rightarrow\infty$. It is also easy to see that $\mathcal{A}^{n}$ is closed by checking that 
each $\mathcal{A}_{\frac{1}{j},j,j}$ is closed. That implies that $Q\in\mathcal{A}^{n}$. 
Finally, we need to check that $Q$ is a simple point process. Let $I_{j,\delta}=[(j-1)\delta,j\delta]$. 
We have for any $Q\in\mathcal{A}^{n}$,
\begin{align}
Q\left(\exists t: N[t-,t]\geq 2\right)&=Q\left(\bigcup_{k=1}^{\infty}\left\{\exists t\in[-k,k]: N[t-,t]\geq 2\right\}\right)
\\
&=Q\left(\bigcup_{k=1}^{\infty}\bigcap_{\delta>0}\bigcup_{j=-[k/\delta]+1}^{[k/\delta]}
\left\{\omega:\#\{\omega\cup I_{j,\delta}\}\geq 2\right\}\right)\nonumber
\\
&\leq\sum_{k=1}^{\infty}\inf_{\delta=\frac{1}{m},m\geq n}\sum_{j=-[k/\delta]+1}^{[k/\delta]}Q(\#\{\omega\cup I_{j,\delta}\}\geq 2)\nonumber
\\
&\leq\sum_{k=1}^{\infty}\inf_{\delta=\frac{1}{m},m\geq n}\{2[k/\delta]\delta g(\delta)\}\nonumber
\\
&=0.\nonumber
\end{align}
Hence, $\mathcal{A}^{n}$ is precompact in our topology.  Since $\mathcal{A}^{n}$ is closed, it is compact.
\end{proof}

\section{Concluding Remarks}\label{conclusion}

In this paper, we obtained a process-level large deviation principle for a wide class of simple point processes,
i.e. nonlinear Hawkes process. Indeed, the methods and ideas should apply to some other simple point processes as well
and we should expect to get the same expression for the rate function $H(Q)$. For $H(Q)<\infty$, it should satisfy
\begin{equation}
H(Q)=\int_{\Omega}\int_{0}^{1}\lambda(\omega,s)-\hat{\lambda}(\omega,s)+\log\left(\frac{\hat{\lambda}(\omega,s)}{\lambda(\omega,s)}\right)
\hat{\lambda}(\omega,s)dsQ(d\omega),
\end{equation}
where $\lambda(\omega,s)$ is the intensity of the underlying simple point process. Now, it would be interesting
to ask for what conditions for a simple point process would guarantee the process-level large deviation principle 
that we obtained in our paper? First, we have to assume that $\lambda(\omega,t)$ is predictable and progressively measurable.
Second, from the proof of the upper bound in our paper, the key assumption we used about
nonlinear Hawkes processes was that $\lim_{z\rightarrow\infty}\frac{\lambda(z)}{z}=0$. That is crucial to guarantee
the superexponential estimates we needed for the upper bound. If for a simple point process, we have $\lambda(\omega,t)\leq F(N(t,\omega))$
for some sublinear function $F(\cdot)$, we would expect the superexponential estimates still works for the upper bound.
Third, it is not enough to have $\lambda(\omega,t)\leq F(N(t,\omega))$ for sublinear $F(\cdot)$ to get the full large deviation
principle. The reason is that in the proof of lower bound, in particular, in Lemma \ref{mainlower}, we need to use the fact
that if $\lambda(\omega,t)$ has memory, the memory will decay to zero eventually over time. For nonlinear Hawkes process,
this is guaranteed by the assumption that $\int_{0}^{\infty}h(t)dt<\infty$, which is crucial in the proof of Lemma \ref{mainlower}.
Indeed for any simple point process $P$, if you want to define $P^{\omega^{-}}$, the probability measure conditional 
on the past history $\omega^{-}$, to make sense of it, you have to have some regularities to ensure that the memory
of the history will decay to zero eventually over time.
From this perspective, nonlinear Hawkes processes form a rich and ideal class for which the process-level large deviation principle holds.

\section*{Acknowledgements}

The author is enormously grateful to his advisor Professor S. R. S. Varadhan for suggesting this topic and for his superb guidance, 
understanding, patience, and generosity. 
The author also wishes to thank an anonymous referee for helpful suggestions. 
The author is supported by NSF grant DMS-0904701, DARPA grant and MacCracken Fellowship at NYU.

\end{document}